\documentclass{amsart}
\usepackage[margin=3.5cm]{geometry}
\usepackage{mathtools}               %provides \coloneqq
\usepackage{amssymb}
\newtheorem{theorem}{Theorem}[section]
\newtheorem{intro}{Theorem}
\newenvironment{introtheorem}[1]
  {\renewcommand\theintro{#1}\intro}
  {\endintro}
\newtheorem{proposition}[theorem]{Proposition}

\newtheorem{lemma}[theorem]{Lemma}
\theoremstyle{remark}
\newtheorem{definition}{Definition}
\newtheorem*{remark}{Remark}
\newtheorem*{notation}{Notation}
\numberwithin{equation}{section}
\newcommand{\bt}{\mathbf t}                             % Thue--Morse sequence

\DeclareMathOperator{\e}{\mathrm{e}}                    % exp(2\pi ix)
\newcommand{\tO}{\mathtt 0}                   % typewriter 0 for the symbol 0
\newcommand{\tL}{\mathtt 1}                   % typewriter 1 for the symbol 1

\DeclareMathOperator{\logp}{\log^+\!}                   % log^+
\allowdisplaybreaks
\begin{document}
%{{{ title, authors, thanks, subject class
\title[Level of distribution of Thue--Morse]%
{The level of distribution of the Thue--Morse sequence}
\author{Lukas Spiegelhofer}
\address{Institute of Discrete Mathematics and Geometry,
Vienna University of Technology,
Wiedner Hauptstrasse 8--10, 1040 Vienna, Austria}

\thanks{
The author acknowledges support by the Austrian Science Fund (FWF), Project F5502-N26,
which is a part of the Special Research Program ``Quasi Monte Carlo methods: Theory and Applications''.
Moreover, the author wishes to acknowledge support by the project MuDeRa,
which is a joint project between the FWF (I-1751-N26) and the ANR (Agence Nationale de la Recherche, France, ANR-14-CE34-0009).
}

%11N69: Distribution of integers in special residue classes
%11B25: Arithmetic progressions [See also 11N13]
%11B85: Automata sequences

%11A63: Radix representation; digital problems {For metric results, see 11K16}
%11K16: Normal numbers, radix expansions, Pisot numbers, Salem numbers, good lattice points, etc. [See also 11A63]
%11B83: Special sequences and polynomials

%}}} title, authors, thanks, subject class
%{{{ abstract
\begin{abstract}
The level of distribution of a complex valued sequence $b$ measures ``how well $b$ behaves'' on arithmetic progressions $nd+a$.
Determining whether $\theta$ is a level of distribution for $b$ involves summing a certain error over $d\leq D$, where $D$ depends on $\theta$;
this error is given by comparing a finite sum of $b$ along $nd+a$ and the expected value of the sum.
We prove that the Thue--Morse sequence has level of distribution $1$, which is essentially best possible.
More precisely, this sequence gives one of the first nontrivial examples of a sequence satisfying a Bombieri--Vinogradov type theorem for each exponent $\theta<1$.
In particular, this result improves on the level of distribution $2/3$ obtained by M\"ullner and the author.

As an application of our method, we show that the subsequence of the Thue--Morse sequence indexed by $\lfloor n^c\rfloor$, where $1<c<2$, is simply normal. That is, each of the two symbols appears with asymptotic frequency $1/2$ in this subsequence.
This result improves on the range $1<c<3/2$ obtained by M\"ullner and the author and closes the gap that appeared when Mauduit and Rivat proved (in particular) that the Thue--Morse sequence along the squares is simply normal.
%The method of proof is a refinement of the method used in the previous papers in this series.
In the proofs, we reduce both problems to an estimate of a certain Gowers uniformity norm of the Thue--Morse sequence similar to that given by Konieczny (2017).

\end{abstract}
%}}} abstract
\maketitle
%{{{ Section Introduction
\section{Introduction}
The Thue--Morse sequence $\bt$ is one of the most easily defined automatic sequences.
Like any automatic sequence, it can be defined using a constant-length substitution over a finite alphabet: $\bt$ is the unique fixed point of the substitution $\tO\mapsto\tO\tL$, $\tL\mapsto\tL\tO$ that starts with $\tO$.
Therefore $\bt=\tO\tL\tL\tO\tL\tO\tO\tL\tL\tO\tO\tL\tO\tL\tL\tO\ldots$.
Alternatively, this sequence can be defined using the \emph{binary sum-of-digits function} $s$, which counts the number of $\tL$s in the binary expansion of a nonnegative integer $n$:
we have $\bt(n)=\tO$ if and only if $s(n)\equiv 0\bmod 2$.
The equivalence of these two definitions can be proved via a third description:
start with the one-element sequence $\bt^{(0)}\coloneqq(\tO)$ and define $\bt^{(n+1)}$ by concatenating $\bt^{(n)}$ and the Boolean complement $\neg\bt^{(n)}$. Then $\bt$ is the pointwise limit of these finite sequences.
In this work, we will adapt the second viewpoint.
In fact, in the proofs we will work with the sequence $(-1)^{s(n)}$ instead of $\bt$.
For an overview on the Thue--Morse sequence, we refer the reader to the article by Allouche and Shallit~\cite{AS1999}, which points out occurrences of this sequence in different fields of mathematics and offers a good bibliography.
Moreover, we wish to mention the paper~\cite{M2001} by Mauduit.
For a comprehensive account of automatic and morphic sequences, see the book~\cite{AS2003} by Allouche and Shallit.

The main topic of this article is the study of $\bt$ along arithmetic progressions and, more generally, along Beatty sequences $\lfloor n\alpha+\beta\rfloor$.
This topic can be traced back at least to Gelfond~\cite{G1968}, who proved the following theorem on the base-$q$ sum-of-digits function $s_q$.
%{{{ introtheorem: thm_gelfond
\begin{introtheorem}{A}[Gelfond]\label{thm_gelfond}
Let $q,m,d,b,a$ be integers and $q,m,d\geq 2$.
Suppose that $\gcd(m,q-1)=1$.
Then
\begin{equation*}%\label{eqn_gelfond}
\bigl \lvert\{
1\leq n\leq x: n\equiv a\bmod d, s_q(n)\equiv b\bmod m
\}\bigr \rvert
=
\frac{x}{dm}+\mathcal O\bigl(x^\lambda\bigr)
\end{equation*}
for some $\lambda<1$ not depending on $x,d,a$, and $b$.
\end{introtheorem}
%}}}
We are particularly interested in the error term for \emph{sparse} arithmetic progressions, having large common difference $d$.
This leads us directly to the other main concept of this paper, the notion of \emph{level of distribution}.
Very roughly speaking, the level of distribution is a measure of how well a given sequence behaves on arithmetic progressions.
A formal definition is given by Fouvry and Mauduit~\cite{FM1996b}, for example, which we adapt here.
%{{{ definition: level of distribution
\begin{definition}
Let $c=(c_n)_{n\geq 0}$ be a sequence of complex numbers, and for each integer $d\geq 1$ let $\mathcal Q(d)$ and $\mathcal R(d)\neq\emptyset$ be subsets of $\mathbb Z/d\mathbb Z$ such that $\mathcal Q(d)\subseteq \mathcal R(d)$.
The sequence $c$ has \emph{level of distribution} $\theta$ with respect to $\mathcal Q$ and $\mathcal R$ if
for all $\varepsilon>0$ and $A>0$ we have for all $x\geq 1$
\begin{multline*}%\label{eqn_exponent_of_distribution_def}
\sum_{1\leq d\leq x^{\theta-\varepsilon}}
\max_{0\leq y\leq x}
\max_{\substack{0\leq a<d\\a+d\mathbb Z\in \mathcal Q(d)}}
\Biggl \lvert
\sum_{\substack{0\leq n<y\\n\equiv a \bmod d}}c_n
-\frac 1{\lvert \mathcal R(d)\rvert}
\sum_{\substack{0\leq n<y\\n+d\mathbb Z\in \mathcal R(d)}}c_n
\Biggr \rvert\\
=\mathcal O\Biggl(\Biggl(\sum_{0\leq n<x}\lvert c_n\rvert\Biggr)(\log 2x)^{-A}\Biggr).
\end{multline*}
The implied constant may depend on $A$ and $\varepsilon$.
Moreover, in this definition the maximum over the empty index set is defined as $0$.
\end{definition}
%}}}
The level of distribution (also called \emph{exponent of distribution} by some authors) is an important concept in sieve theory.
As a striking application, a variant of this concept was used in the ``bounded gaps between primes'' paper by Zhang~\cite{Z2014}. %(we will come back to this after Theorem~\ref{thm_BV}).
For more information on this subject, we refer the reader to the survey by Kontorovich~\cite{K2014}.
Moreover, we wish to draw the attention of the reader to the book~\cite{FI2010}  on sieve theory by Friedlander and Iwaniec, in particular Chapter~22 on the level of distribution.

We are ready to present our main result.
%{{{ Main theorem
\begin{theorem}\label{thm_intromain}
The Thue--Morse sequence has level of distribution $1$ with respect to $\mathcal Q$ and $\mathcal R$ given by $\mathcal Q(d)=\mathcal R(d)=\mathbb Z/d\mathbb Z$.
More precisely, for all $\varepsilon>0$ we have

\[
\sum_{1\leq d\leq D}
\max_{0\leq y\leq x}
\max_{0\leq a<d}
\Biggl \lvert
\sum_{\substack{0\leq n<y\\n\equiv a\bmod d}}
(-1)^{s(n)}%-\frac y{2d}
\Biggr \rvert
=\mathcal O(x^{1-\eta})
\]
for some $\eta>0$ depending on $\varepsilon$, where $D=x^{1-\varepsilon}$.
\end{theorem}
%}}}
Before presenting some history, we wish to say a word about the proof: we are going to reduce the problem to the estimation of a certain \emph{Gowers uniformity norm} of the Thue--Morse sequence.
These expressions appear by repeated application of van der Corput's inequality and have the form
\[
\sum_{\substack{0\leq n<2^\rho\\0\leq r_1,\ldots,r_m< 2^\rho}}
\prod_{\varepsilon\in\{0,1\}^m}
(-1)^{s_\rho(n+\varepsilon\cdot r)},
\]
where $\varepsilon\cdot r=\sum_{1\leq i\leq m}\varepsilon_ir_i$ and $s_\rho$ is the truncated sum-of-digits function in base $2$ defined by $s_\rho(n)=s(n\bmod 2^\rho)$.
The proof of a very similar statement was given recently by Konieczny~\cite{K2017}, and we use ideas from that paper to prove our estimate.

In order to put Theorem~\ref{thm_intromain} into context, we present some related theorems.
The well-known Bombieri--Vinogradov theorem concerns the level of distribution of the von Mangoldt function $\Lambda$, which is defined by $\Lambda(n)=\log p$ if $n=p^k$ for some prime $p$ and some $k\geq 1$ and $\Lambda(n)=0$ otherwise.
This theorem states that $\Lambda$ has level of distribution $1/2$ with respect to $\mathcal Q$ and $\mathcal R$ given by $\mathcal Q(d)=\mathcal R(d)=(\mathbb Z/d\mathbb Z)^*$.
%{{{ Introtheorem: thm_BV
\begin{introtheorem}{B}[Bombieri--Vinogradov]\label{thm_BV}
Let $d\geq 1$ and $a$ be integers and define
\[ \psi(x;d,a) = \sum_{\substack{1\leq n\leq x\\n\equiv a\bmod d}}\Lambda(n). \]
For all real numbers $A>0$ there exist $B>0$ and a constant $C$ such that setting $D=x^{1/2}(\log x)^{-B}$ we have for all $x\geq 2$
\[
\sum_{1\leq d\leq D}
\max_{1\leq y\leq x}
\max_{\substack{0\leq a<d\\\gcd(a,d)=1}}
\left\lvert\psi(y;d,a)-\frac y{\varphi(d)}\right\rvert
\leq Cx(\log x)^{-A}.
\]
Here $\varphi$ denotes Euler's totient function.
\end{introtheorem}
%}}}
No improvement on the level of distribution $1/2$ in this theorem is currently known; meanwhile the Elliott--Halberstam conjecture~\cite{EH1970} states that we can choose $D=x^{1-\varepsilon}$ for any $\varepsilon>0$. That is, it is conjectured that the primes have level of distribution $1$.
Improvements on the exponent $1/2$ exist for certain sequences of integers;
we refer to the articles~\cite{F1982,F1984} by Fouvry, by Fouvry and Iwaniec~\cite{FI1980} and by Friedlander and Iwaniec~\cite{FI1985}.
Moreover, we note the series~\cite{BFI1986,BFI1987,BFI1989} by Bombieri, Friedlander and Iwaniec concerning these questions.
In this context, we also note the result of Goldston, Pintz, and Y\i ld\i r\i m~\cite{GPY2009}, who showed in particular the following conditional result: if the primes have level of distribution $\theta$ for some $\theta>1/2$, there is a constant $C$ such that $p_{n+1}-p_n<C$ infinitely often, where $p_n$ is the $n$-th prime.
In a groundbreaking paper we mentioned before, Zhang~\cite{Z2014}
used the Goldston--Pintz--Y\i ld\i r\i m method and a variant of the Bombieri--Vinogradov theorem to prove the above result unconditionally.
Maynard~\cite{M2015} later proved the bounded gaps result using only the classical Bombieri--Vinogradov theorem.
%Fouvry, Kowalski and Michel~\cite{FKM2015}.

Improvements on the level $1/2$ are also known for the sum-of-digits function modulo $m$.
Fouvry and Mauduit~\cite{FM1996} established $0.5924$ as level of distribution of the Thue--Morse sequence, with respect to $\mathcal Q$ and $\mathcal R$, where $\mathcal Q(d)=\mathcal R(d)=\mathbb Z/d\mathbb Z$.
%{{{ Introtheorem: Fouvry--Mauduit 1996
\begin{introtheorem}{C}[Fouvry--Mauduit]\label{thm_FM1996}
Set
\[
A(x;d,a)=\bigl \lvert\bigl\{0\leq n<x:\bt(n)=\tO, n\equiv a\bmod d\bigr\}\bigr \rvert.
\]
Then
%{{{ eqn_FM96
\begin{equation}\label{eqn_FM96}
\sum_{1\leq d\leq D}
\max_{1\leq y\leq x}
\max_{0\leq a<d}
\Bigl \lvert  A(y;d,a)-\frac y{2d} \Bigr \rvert \leq Cx(\log 2x)^{-A}
\end{equation}
%}}} eqn_FM96
for all real $A$ and $D=x^{0.5924}$.
\end{introtheorem}
%}}}
More generally, for $m\geq 2$ they also study the sum-of-digits function in base $2$ modulo $m$, obtaining the weaker level of distribution $0.55711$.
Using sieve theory, they apply this result to the study of the sum of digits modulo $m$ of numbers having at most two prime factors.
Later, Mauduit and Rivat~\cite{MR2010}, in an important paper, managed to treat the sum of digits modulo $m$ of prime numbers, thereby answering one of the questions posed by Gelfond~\cite{G1968}.

M\"ullner and the author~\cite{MS2017} improved the exponent $0.5924$ to $2/3-\varepsilon$, thereby establishing $2/3$ as an admissible level of distribution of the Thue--Morse sequence.

Fouvry and Mauduit~\cite{FM1996b} also considered, more generally, the sum-of-digits function $s_q$ in base $q$ modulo an integer $m$ such that $\gcd(m,q-1)=1$.
They obtain the remarkable result that the level of distribution approaches $1$ as the base $q$ gets larger.

%{{{ Introtheorem D
\begin{introtheorem}{D}[Fouvry--Mauduit]\label{thm_FM1996b}
Let $q\geq 2$, $m\geq 1$ and $b$ be integers such that $\gcd(m,q-1)=1$.
Then for all $A$ and $\varepsilon>0$ we have for all $x\geq 1$
%{{{ eqn_FM96b
\begin{equation*}%\label{eqn_FM96b}
\sum_{1\leq d\leq x^{\theta_q-\varepsilon}}
\max_{0\leq y\leq x}
\max_{0\leq a<d}
\Biggl \lvert
  \sum_{\substack{n<y,s_q(n)\equiv b\bmod m\\n\equiv a\bmod d}}1-
\frac 1d\sum_{n<y,s_q(n)\equiv b\bmod m}1\Biggr \rvert =\mathcal O(x(\log 2x)^{-A}),
\end{equation*}
%}}} eqn_FM96b
where $\theta_q\rightarrow 1$ as $q\rightarrow\infty$.
The implied constant depends at most on $m$, $q$, $A$ and $\varepsilon$.
\end{introtheorem}
%}}}
As an application of this theorem, they study the sum
$\sum_{n<x,s_q(n)\equiv b\bmod m}\Lambda_k(n)$, where $\Lambda_k$ is the generalized von Mangoldt function of order $k\geq 1$ (\cite[Corollaire~2]{FM1996b}).

Theorem~\ref{thm_FM1996b} motivates us to ask which sequences have level of distribution equal to $1$.
In the above-cited paper by Fouvry and Mauduit~\cite{FM1996b}, for example, a list of sequences having this property is given.
Moreover, we note~\cite[Chapter~22.3]{FI2010}, which studies the level of distribution for additive convolutions, giving further examples.
However, in these examples, other than the trivial example $c_n=1$ for all $n$,
the maximum over $a$ does not play a r\^ole: the set $\mathcal Q(d)$ consists of at most one element.

We are interested in sequences $c$ having level of distribution $1$ and such that the set $\mathcal Q(d)$ contains ``many'' residue classes.
In other words, we want to find analogues of the Elliott--Halberstam conjecture.
Requiring monotonicity of $c$, examples can be constructed easily:
$c(n)=n$ is such an example, and more generally, increasing sequences $c$ satisfying certain growth conditions have this property.
Apart from such ``trivial'' sequences, no other examples seem to be known.
Our Theorem~\ref{thm_intromain}, giving such an example, is therefore of considerable significance.

Moreover, we note that our method can certainly be adapted to $s_q(n)\bmod m$ for general bases $q\geq 2$ and $m\geq 1$, which yields $\theta_q=1$ in Theorem~\ref{thm_FM1996b}.

The second focus of this paper concerns \emph{Piatetski-Shapiro sequences},
which are sequences of the form $(\lfloor n^c\rfloor)_{n\geq 0}$ for some $c\geq 1$.
For stating the second main theorem, we do not need additional preparation.
%{{{ Theorem
\begin{theorem}\label{thm_intronc}
Let $1<c<2$.
The Thue--Morse sequence along $\lfloor n^c\rfloor$ is simply normal.
That is, each of the letters $\tO$ and $\tL$ appears with asymptotic frequency $1/2$ in $n\mapsto\bt(\lfloor n^c\rfloor)$.
\end{theorem}
%}}}
By the argument given in our earlier paper~\cite{MS2017} with M\"ullner, this theorem is proved via a Beatty sequence variant of Theorem~\ref{thm_intromain}.
That theorem in turn is proved by arguments analogous to the arguments in the proof of Theorem~\ref{thm_intromain}, and reduces to the same estimate of the Gowers uniformity norm of Thue--Morse.
Theorem~\ref{thm_intronc} is therefore an application of the method of proof of Theorem~\ref{thm_intromain}.

Again, we present some historical background.
Studying Piatetski--Shapiro subsequences of a given sequence can be seen as a step towards proving theorems on polynomial subsequences.
For example, it is unknown whether there are infinitely many primes of the form $n^2+1$; therefore it is of interest to consider primes of the form $\lfloor n^c\rfloor$ for $1<c<2$ and prove an asymptotic formula for the number of such primes.
Piatetski-Shapiro~\cite{P1953} proved such a formula for $1<c<12/11$, and the currently best known bound is $1<c<2817/2426$ due to Rivat and Sargos~\cite{RS2001}.
In an analogous way, the study of the sum-of-digits function along $\lfloor n^c\rfloor$ was motivated. It is another problem posed by Gelfond~\cite{G1968} to study the distribution of the sum of digits of polynomial sequences in residue classes.
Since this problem could not be solved at first, Mauduit and Rivat~\cite{MR1995, MR2005} considered $q$-multiplicative functions along $\lfloor n^c\rfloor$ (where a $q$-multiplicative function $f:\mathbb N\rightarrow\{z\in\mathbb C:\lvert z\rvert=1\}$ satisfies $f(aq^k+b)=f(aq^k)f(b)$ for nonnegative integers $a,b,k$ such that $b<q^k$)
and they obtained an asymptotic formula for $c<7/5$.
%{{{ introtheorem E
\begin{introtheorem}{E}[Mauduit--Rivat]\label{thm_MR2005}
Let $1<c<7/5$ and set $\gamma=1/c$.
For all $\delta\in(0,(7-5c)/9)$ there exists a constant $C>0$ such that for all $q$-multiplicative functions $f:\mathbb N\rightarrow\{z\in\mathbb C:\lvert z\rvert=1\}$ and all $x\geq 1$ we have
\[
\left\lvert
\sum_{1\leq n\leq x}f\left(\left\lfloor n^c\right\rfloor\right)
-\sum_{1\leq m\leq x^c}\gamma m^{\gamma-1}f(m)
\right\rvert
\leq Cx^{1-\delta}.
\]
\end{introtheorem}
%}}}
Since the Thue--Morse sequence is $2$-multiplicative, it follows in particular that the the subsequence indexed by $\lfloor n^c\rfloor$ assumes each of the two values $\tO$, $\tL$ with asymptotic frequency $1/2$, as long as $1<c<7/5$.
This means that this subsequence is \emph{simply normal}.
In the paper~\cite{DDM2012} by Deshouillers, Drmota, and Morgenbesser, a statement as in Theorem~\ref{thm_MR2005} for arbitrary automatic sequences and $1<c<7/5$ is proved.

Some progress on Gelfond's question on polynomials was made by Drmota and Rivat~\cite{DR2005} and by Dartyge and Tenenbaum~\cite{DT2006};
finally, Mauduit and Rivat~\cite{MR2009} managed to answer Gelfond's question for the polynomial $n^2$.
This latter paper was generalized by Drmota, Mauduit and Rivat~\cite{DMR2018}, who showed that
in fact $\bt(n^2)$ defines a \emph{normal sequence}, by which we understand an infinite sequence on $\{\tO,\tL\}$ such that every finite sequence of length $k$ occurs as a factor (contiguous finite subsequence) with asymptotic frequency $2^{-k}$.
This result also generalizes a paper by Moshe~\cite{M2007} who showed that every finite word on $\{\tO,\tL\}$ occurs as a factor of $n\mapsto \bt(n^2)$ at least once.

However, the distribution of the sum of digits of $\lfloor n^c\rfloor$ in residue classes, for %(
$c\in[1.4,2)$, %]
remained an open problem.
Progress in this direction was made by the author~\cite{S2014},
who improved the bound on $c$ to $ 1<c\leq 1.42$ for the Thue--Morse sequence.
The key idea in that paper is to approximate $\lfloor n^c\rfloor$ by a Beatty sequence $\lfloor n\alpha+\beta\rfloor$ and thus reduce the problem to a linear one.

M\"ullner and the author~\cite{MS2017}, using the same linearization argument
and a Bombieri--Vinogradov type theorem for the Thue--Morse sequence on Beatty sequences, were able to extend this range to $1<c<3/2$.
Moreover, we could handle occurrences of factors in Piatetski-Shapiro subsequences of $\bt$, thus showing that $\bt\left(\lfloor n^c\rfloor\right)$ defines a normal sequence for $1<c<3/2$.
%{{{ Introtheorem F
\begin{introtheorem}{F}[M\"ullner--Spiegelhofer]\label{thm_normality}
Let $1<c<3/2$. Then the sequence $\mathbf u=\bigl(\bt\left(\left\lfloor n^c\right\rfloor\right)\bigr)_{n\geq 0}$ is normal.
More precisely, for any $L\geq 1$ there exists an exponent $\eta>0$ and a constant $C$
such that
\[
\Bigl \lvert \left\lvert\bigl\{n<N:\mathbf u(n+i)=\omega_i\textrm{ for }0\leq i<L\bigr\}\right\rvert-N/2^L\Bigr \rvert
\leq CN^{1-\eta}
\]
for all  $\bigl(\omega_0,\ldots,\omega_{L-1}\bigr)\in\{\tO,\tL\}^L$.
\end{introtheorem}
%}}}
This theorem also improved on an earlier result by the author~\cite{S2015}, who obtained normality for $1<c<4/3$, using an estimate for Fourier coefficients related to the Thue--Morse sequence provided by Drmota, Mauduit and Rivat~\cite{DMR2018}.

Our Theorem~\ref{thm_intronc} finally closes the gap in the set of exponents $c$ such that we have an asymptotic formula for Thue--Morse on $\lfloor n^c\rfloor$.
This gap appeared with the Mauduit--Rivat result on squares; at that time, the gap was     %(
$[1.4,2)$, %]
now it was only left to close the smaller gap %(
$[1.5,2)$.                                    %]

However, the case $c>2$ remains open for now, for $c\in\mathbb Z$ (which is contained in Gelfond's problem on polynomial subsequences) as well as for Piatetski-Shapiro sequences.
For example, it is a notorious open question to prove that $\tO$ occurs with frequency $1/2$ in $n\mapsto \bt(n^3)$.
(If this result is proved some day, there will be a new gap to be closed.)

Mauduit~\cite[Conjecture~1]{M2001} conjectures that
\[
\lim_{N\rightarrow\infty}\frac 1N
\left\{
1\leq n\leq N:s_q(\lfloor n^c\rfloor)\equiv b\bmod m
\right\}
=\frac 1m
\]
for almost all $c>1$, where $q\geq 2$, $m\geq 1$ and $b$ are integers.
While this almost-all result is known for $1<c<2$, as he notes just before this conjecture, we believe (as we noted before) that our method can be adapted to generalize our results to general sequences $s_q(n) \bmod m$ and thus to prove the asymptotic identity for all $c\in(1,2)$.
However, while we are confident that the asymptotic identity in Mauduit's conjecture holds for \emph{all} non-integer $c>1$, the case $c>2$ cannot yet be handled by our methods.

Moreover, we note that it would be interesting to generalize the normality result from Theorem~\ref{thm_normality} to all exponents $1<c<2$.
%{{{notation
\begin{notation}
For a real number $x$, we write $\e(x)=\exp(2\pi i x)$, $\{x\}=x-\lfloor x\rfloor$, $\lVert x\rVert=\min_{n\in\mathbb Z}\lvert x-n\rvert$ and $\langle\cdot\rangle=\lfloor x+1/2\rfloor$ (the ``nearest integer'' to $x$).
For a prime number $p$ let $\nu_p(n)$ be the exponent of $p$ in the prime factorization of $n$.
We define the \emph{truncated binary sum-of-digits function}
\[    s_\lambda(n)\coloneqq s(n'),    \]
where $0\leq n'<2^\lambda$ and $n'\equiv n\bmod 2^\lambda$,
which is the $2^\lambda$-periodic extension of the restriction of $s$ to $\{0,\ldots,2^\lambda-1\}$.
For $\mu\leq \lambda$ we define the \emph{two-fold restricted binary sum-of-digits function}
\[    s_{\mu,\lambda}(n)=s_\lambda(n)-s_\mu(n).    \]
For a real number $x\geq 0$, we set
\[    \logp x=\max\left\{1,\log x\right\}.    \]
The symbol $\mathbb N$ denotes the set of nonnegative integers.
\end{notation}
%}}}notation
%}}} Section Introduction
%{{{ Section Results
\section{Results}
In order to (re)state our main theorem, we introduce some notation.
Let $\alpha,\beta,y$ and $z$ be nonnegative real numbers such that $\alpha\geq 1$.
We define
\[
A(y,z;\alpha,\beta)=\bigl \lvert\bigl\{y\leq m<z:\bt(m)=\tO\text{ and } \exists n\in\mathbb Z\text{ such that }m=\lfloor n\alpha+\beta\rfloor\bigr\}\bigr \rvert
.
\]
For integers $d=\alpha$ and $a=\beta$, we clearly have
\[
A(y,z;d,a)=\bigl \lvert\bigl\{y\leq m<z:\bt(m)=\tO\textrm{ and } m\equiv a\bmod d\bigr\}\bigr \rvert
.
\]

Our main theorem is the following result, analogous to the Elliott--Halberstam conjecture.
%{{{main theorem
\begin{theorem}\label{thm_1}
The Thue--Morse sequence has level of distribution $1$.
More precisely, for all $\varepsilon>0$ there exist $\eta>0$ and $C$ such that
\[
\sum_{1\leq d\leq D}\max_{\substack{y,z\\0\leq y\leq z\\z-y\leq x}}\max_{0\leq a<d}
\left\lvert  A(y,z;d,a) - \frac{y}{2d} \right\rvert    \leq Cx^{1-\eta}
\]
for $x\geq 1$ and $D=x^{1-\varepsilon}$.
\end{theorem}
%}}}main theorem
Note that this theorem allows intervals %(
$[y,z)$ %]
for arbitrary $y\geq 0$, which is more general than our definition of a level of distribution.

%Very vaguely speaking, this reflects the observation that the Thue--Morse sequence ``everywhere looks the same''.
%A related statement is the fact that the Thue--Morse sequence is uniformly recurrent (see~\cite{AS2003}).

Our second result concerns Piatetski-Shapiro subsequences of the Thue--Morse sequence.
%{{{ Theorem thm_nc
\begin{theorem}\label{thm_nc}
Let $1<c<2$. Then the sequence $n\mapsto \bt(\lfloor n^c\rfloor)$ is simply normal.
More precisely, there exists an exponent $\eta>0$ and a constant $C$
such that
\[
\left\lvert\frac 1N\bigl \lvert\bigl\{0\leq n<N:\bt(\lfloor n^c\rfloor)=\tO\bigr\}\bigr \rvert-\frac 12\right\rvert
\leq CN^{-\eta}.
\]
\end{theorem}

%}}}
For proving this theorem, we follow the general argument presented in Section~4.2 of~\cite{MS2017}.
This argument uses linear approximation of $\lfloor n^c\rfloor$ by $\lfloor n\alpha+\beta\rfloor$ and thus reduces the problem to Beatty sequences.
Therefore Theorem~\ref{thm_nc} is a corollary of the following Beatty sequence version of a statement on the level of distribution.
%{{{ thm_2
\begin{theorem}\label{thm_2}
Let $0<\theta_1\leq\theta_2<1$.
There exist $\eta>0$ and $C$ such that
\[
\int_{D}^{2D}\max_{\substack{y,z\\0\leq y\leq z\\z-y\leq x}}\max_{\beta\geq 0}
\left\lvert  A(y,z;\alpha,\beta) - \frac{z-y}{2\alpha}  \right\rvert
\,\mathrm d\alpha
\leq Cx^{1-\eta}
\]
for all $x$ and $D$ such that $x\geq 1$ and $x^{\theta_1}\leq D\leq x^{\theta_2}$.
\end{theorem}
%}}} thm:2
%TODO: small moduli for Beatty sequences?
In order to derive Theorem~\ref{thm_nc} from this result, it is essential that we have the maximum over $\beta$ inside the integral over $\alpha$,
since we need to approximate $\lfloor n^c\rfloor$ by inhomogeneous (shifted) Beatty sequences $\lfloor n\alpha+\beta\rfloor$.

Concerning Theorem~\ref{thm_1}, a version of this result without the maximum over $a$ follows from work of Martin, Mauduit and Rivat, as we show now.
\begin{remark}
Martin, Mauduit and Rivat~\cite[Proposition~3]{MMR2014} proved an estimate of a sum of type II
containing the following special case:
let $a_m$ and $b_n$ be complex numbers satisfying $\lvert a_m\rvert\leq 1$ and $\lvert b_n\rvert\leq 1$.
Assume that $x\geq 2$, $0<\varepsilon\leq 1/2$, $x^\varepsilon\leq M,N\leq x$ and $MN\leq x$. Then
\[S_0=
\sum_{M<m\leq 2M}\sum_{\substack{N<n\leq 2N\\mn\leq x}}
a_mb_n
(-1)^{s(mn)}
\ll x^{1-\eta}
\]
for an absolute implied constant and some $\eta>0$ only depending on $\varepsilon$.
By dyadic decomposition and using the trivial estimate for $n<x^\varepsilon$, we obtain
\[
\sum_{M<m\leq 2M}
\Biggl \lvert
\sum_{\substack{0\leq n\leq 2N\\mn\leq x}}
(-1)^{s(mn)}
\Biggr \rvert
\ll x^{1-\eta}\log N+Mx^\varepsilon
\]
for $M$ and $N$ satisfying the same restrictions, and with an implied constant that may depend on $\varepsilon$.
Let $x$ be given and assume that $x^\varepsilon\leq M\leq x^\theta$ for some $\theta\in(1/2,1)$.
Set $\varepsilon=\frac{1-\theta}2\leq 1/2$ and $N=x/M$.
Then $N\geq x^\varepsilon$ and
the condition $mn\leq x$ implies $n\leq 2N$.
We obtain

\[\sum_{M<m\leq 2M}
\Biggl \lvert
\sum_{\substack{0\leq k\leq x\\k\equiv 0\bmod m}}
(-1)^{s(k)}
\Biggr \rvert
\ll x^{1-\eta}\log x+Mx^{\varepsilon}.
\]

We use dyadic decomposition again (in $m$), moreover Fouvry and Mauduit~\cite{FM1996} in order to handle residue classes having small modulus $m$, that is, $m\leq x^\varepsilon$.
Moreover, we note (as we did in~\cite{MS2017}) that the error term in their
estimate~\cite[(1.6)]{FM1996} is in fact $x^{1-\eta}$ for some $\eta>0$, which follows from their Th\'eor\`eme~2.
We obtain
\begin{equation*}%\label{eqn_weak_lod}
\sum_{1\leq d\leq D}
\Biggl \lvert
\sum_{\substack{0\leq n\leq x\\n\equiv 0\bmod d}}
  (-1)^{s(n)}
\Biggr \rvert
\leq C x^{1-\eta}
\end{equation*}
for $D=x^\theta$ and some $\eta>0$ and $C$ depending on $\theta$.
This is a weak version of a statement of the type ``the Thue--Morse sequence has level of distribution $1$'', where $\mathcal Q(d)$ has only one element.
(We note that we could also handle the maximum over $y\leq x$, using the factor $\e(\beta mn)$ that appears in~\cite[Proposition~3]{MMR2014}.)
The additional value of our paper lies in the maximum over the residue classes modulo $d$.

\end{remark}

Finally, we note the following open questions concerning Theorems~\ref{thm_1} and~\ref{thm_nc}:
\begin{enumerate}
\item In Theorem~\ref{thm_1}, can we choose $D=x(\log x)^{-B}$ for some $B>0$ (using $x(\log x)^{-A}$ as error term)?
\item Does Theorem~\ref{thm_nc} hold for $\lfloor x^2(\log x)^{-C}\rfloor$ (and similar sequences, possibly with a worse error term)?
\end{enumerate}

\noindent\textbf{Plan of the paper.}
In Section~\ref{sec_aux} we state two results (Propositions~\ref{prp_1} and~\ref{prp_2}) from which Theorems~\ref{thm_1} and~\ref{thm_2} follow, moreover an important Gowers uniformity norm estimate of the Thue--Morse sequence, Proposition~\ref{prp_gowers}. We also give an idea of the proof of Proposition~\ref{prp_1}.
In Section~\ref{sec_lemmas} we state lemmas needed for proving the results from Section~\ref{sec_aux}.
Section~\ref{sec_proofs} is devoted to proving Propositions~\ref{prp_1} and~\ref{prp_2}.
Finally, we prove Proposition~\ref{prp_gowers} and a technical lemma appearing in the proof of Propositions~\ref{prp_1} and~\ref{prp_2}.
%}}} Section Results
%{{{ Section Auxiliary results
\section{Auxiliary results}\label{sec_aux}
As in our earlier paper with M\"ullner (\cite[Section~4.1]{MS2017}, and using Fouvry and Mauduit \cite[Th\'eor\`eme~2]{FM1996} for handling small $d$), it is sufficient to prove the following two results in order to obtain our main theorems.
%{{{ prp_1
\begin{proposition}\label{prp_1}
For real numbers $N,D\geq 1$ and $\xi$ set
\begin{equation}\label{eqn_S0_def}
S_0 = S_0(N,D,\xi)
=
\sum_{D\leq d<2D}
\max_{a\geq 0}
\Biggl \lvert \sum_{0\leq n<N}
  \e\left(\frac 12s(nd+a)\right)\e(n\xi)
\Biggr \rvert
. \end{equation}
Let $\rho_2\geq \rho_1>0$.
There exists an $\eta>0$ and a constant $C$ such that
%{{{ eqn_prp1_estimate
\begin{equation*}%\label{eqn_prp1_estimate}
\frac{S_0}{ND} \leq CN^{-\eta}
\end{equation*}
%}}} eqn_prp1_estimate
holds for all $\xi\in\mathbb R$ and all real numbers $N,D\geq 1$ satisfying
$N^{\rho_1}\leq D \leq N^{\rho_2}$.
\end{proposition}
%}}} prp_1
%{{{ prp_2
\begin{proposition}\label{prp_2}
For real numbers $D,N\geq 1$ and $\xi$ set
\begin{equation}\label{eqn_S0_def_continuous}
S_0 = S_0(N,D,\xi)
=
\int_{D}^{2D}
\max_{\beta\geq 0}
\Biggl \lvert
    \sum_{0\leq n<N}\e\left(\frac 12s\bigl(\lfloor n\alpha+\beta\rfloor\bigr)\right)\e(n\xi)
\Biggr \rvert
\,\mathrm d \alpha
.
\end{equation}
Let $\rho_2\geq \rho_1>0$.
There exist $\eta>0$ and a constant $C$ such that
%{{{ eqn_prp2_estimate
\begin{equation*}%\label{eqn_prp2_estimate}
\frac{S_0}{ND}
\leq
CN^{-\eta}
\end{equation*}
%}}} eqn_prp2_estimate
holds for all real numbers $D,N\geq 1$ satisfying
$N^{\rho_1}\leq D \leq N^{\rho_2}$ and for all $\xi\in\mathbb R$.
\end{proposition}
%}}} prp:2
In the proof of these results, we will use the following essential \emph{Gowers uniformity norm} estimate of the Thue--Morse sequence (see Konieczny~\cite{K2017}).

%{{{ prp_gowers
\begin{proposition}\label{prp_gowers}
Let $m\geq 2$ be an integer.
There exists some $\eta>0$ and some $C$ such that
\[
\frac 1{2^{(m+1)\rho}}
\sum_{\substack{0\leq n<2^\rho\\0\leq r_1,\ldots,r_m< 2^\rho}}
%g_{\rho}(r,\mathbf a)
\e\left(\frac 12\sum_{\varepsilon\in\{0,1\}^m}s_\rho(n+\varepsilon\cdot r)\right)
\leq C 2^{-\rho\eta}
\]
for all $\rho\geq 0$, where $\varepsilon\cdot r=\sum_{1\leq i\leq m}\varepsilon_ir_i$.

\end{proposition}
%}}} prp_gowers

We wish to give a rough idea of the proof of Proposition~\ref{prp_1}
(Proposition~\ref{prp_2} being proved essentially in the same way.)

\noindent\textbf{Idea of the proof of Proposition~\ref{prp_1}.}
The key idea is to reduce the number of digits that have to be taken into account, and thus to replace the sum-of-digits function $s$ by its truncated version $s_\rho$.
Here $2^\rho$ will be significantly smaller than $N$, so that (we simplify things a bit to convey the idea) we may replace the sum over $s(nd+a)$ by a full sum over the periodic function $s_\rho(n)$.
This reducing of the digits  is achieved by a refinement of the method used by M\"ullner and the author~\cite{MS2017}, which in turn builds on the ideas from the papers~\cite{MR2009,MR2010} by Mauduit and Rivat.

First, we apply van der Corput's inequality and use a ``carry propagation lemma'' in order to replace $s$ by $s_\lambda$.
In general, $2^\lambda$ will be much larger than $N$, so that we have to reduce $\lambda$ further.
The next step is to apply the generalized van der Corput inequality repeatedly. With each application, we remove $\mu$ many digits.
This is achieved by appealing to the Dirichlet approximation theorem, by which we can find a multiple of $\alpha=d/2^{j\mu}$ that is close to a multiple of $2^\mu$. This property can be used to discard the $\mu$ lowest digits.

By this repeated application the estimate is reduced to an estimate of a so-called Gowers uniformity norm of the Thue--Morse sequence; a related estimate was recently given by Konieczny~\cite{K2017}.
%}}} Section Auxiliary results
%{{{ Section sec_lemmas
\section{Lemmas}\label{sec_lemmas}
We have the following series of lemmas that can also be found in our earlier paper with M\"ullner~\cite{MS2017}.

The first lemma can be proved by elementary considerations.
%{{{ lem_fractional_part_facts
\begin{lemma}\label{lem_fractional_part_facts}
Let $a,b\in\mathbb R$ and $n\in\mathbb N$.
\begin{align}
&\text{If $\lVert a\rVert<\varepsilon$ and $\lVert b\rVert\geq \varepsilon$, then
$\lfloor a+b\rfloor=\langle a\rangle+\lfloor b\rfloor$.}\label{item_fpf_1}\\
&\lVert na\rVert \leq n\lVert a\rVert\label{item_fpf_2}.\\
&\text{If $\lVert a\rVert<\varepsilon$ and $2n\varepsilon<1$, then
$\langle na\rangle=n\langle a\rangle$.} \label{item_fpf_3}
\end{align}
\end{lemma}
%}}} lem_fractional_part_facts
As an essential tool, we will use repeatedly the following generalized van der Corput inequality~\cite[Lemme 17]{MR2009}.
%{{{ lem_vdc (van der Corput's inequality)
\begin{lemma}\label{lem_vdc}
Let $I$ be a finite interval containing $N$ integers and
let $z_n$ be a complex number for $n\in I$.
For all integers $K\geq 1$ and $R\geq 1$ we have
\begin{equation}\label{eqn_vdc}
  \Biggl \lvert\sum_{n\in I}z_n\Biggr \rvert^2
\leq
  \frac{N+K(R-1)}R
  \sum_{0\leq\lvert r\rvert<R}\left(1-\frac {\lvert r\rvert}R\right)
  \sum_{\substack{n\in I\\n+Kr\in I}}z_{n+Kr}\overline{z_n}
.
\end{equation}
%In particular, the right hand side is a nonnegative real number.
\end{lemma}
%}}} lem_vdc (van der Corput's inequality)
Assume that $\alpha$ is a real number and $N$ is a nonnegative integer.
We define the \emph{discrepancy} of the sequence $n\alpha$ modulo $1$:
%{{{ definition of discrepancy
\[
D_N(\alpha)=\sup_{\substack{0\leq x\leq 1\\y\in\mathbb R}}  %(
\left\lvert  \frac 1N\sum_{n<N}c_{[0,x)+y+\mathbb Z}(n\alpha)-x  \right\rvert. %]
\]
%}}} definition of discrepancy
Applying this definition, using $x=1/(KT)$ and $\alpha/K$ instead of $\alpha$, we obtain the following lemma.
%{{{ lem_f_discrepancy (Discrepancy estimate for f(n))
\begin{lemma}\label{lem_f_discrepancy}
Let $J$ be an interval in $\mathbb R$ containing $N$ integers
and let $\alpha$ and $\beta$ be real numbers.
Assume that $t,T,k$ and $K$ are integers such that $0\leq t<T$ and $0\leq k<K$.
Then
%{{{
\[
\left\lvert\bigl\{
    n\in J: \frac tT\leq \{n\alpha+\beta\}<\frac{t+1}T,
    \lfloor n\alpha+\beta\rfloor\equiv k\bmod K
\bigr\}\right\rvert
=
\frac N{KT} + O\left(ND_N\left(\frac \alpha K\right)\right)
\]
%}}}
with an absolute implied constant.
\end{lemma}
%%}}} lem_f_discrepancy (Discrepancy estimate for f(n))
In the estimation of our error terms, we will use the following mean discrepancy results.
%{{{ lem_mean_discrepancy
\begin{lemma}\label{lem_mean_discrepancy}
For integers $\mu\geq 0$ and $N\geq 1$ we have
%{{{ eqn_mean_discrepancy_sum
\begin{equation*}%\label{eqn_mean_discrepancy_sum}
\sum_{0\leq d<2^{\mu}}
D_N\left(\frac d{2^\mu}\right)
\ll
\frac{N+2^\mu}{N}(\logp N)^2.
\end{equation*}
%}}} eqn_mean_discrepancy_sum
Moreover, the estimate
%{{{ eqn_mean_discrepancy_int
\begin{equation*}%\label{eqn_mean_discrepancy_int}
\int_0^1 D_N(\alpha)\,\mathrm d\alpha
\ll
\frac{(\logp N)^2}{N}
\end{equation*}
%}}} eqn_mean_discrepancy_int
holds. The implied constants in these estimates are absolute.
\end{lemma}
%}}} lem_mean_discrepancy
The following ``carry propagation lemma''  will allow us to replace the sum-of-digits function $s$ by its truncated version $s_\lambda$.
Statements of this type were used by Mauduit and Rivat in their papers on the sum of digits of primes and squares~\cite{MR2009,MR2010}.
%{{{ lem_carry (carry propagation lemma)
\begin{lemma}\label{lem_carry}
Let $r,N,\lambda$ be nonnegative integers and $\alpha>0,\beta\geq 0$ real numbers.
Assume that $I$ is an interval containing $N$ integers.
Then
%{{{ carry propagation
\begin{multline*}
  \bigl \lvert
    \left\{
      n\in I:
     s\bigl(\lfloor(n+r)\alpha+\beta\rfloor\bigr)-s\bigl(\lfloor n\alpha+\beta\rfloor\bigr)
\neq
      s_\lambda\bigl(\lfloor(n+r)\alpha+\beta\rfloor\bigl)-s_\lambda\bigl(\lfloor n\alpha+\beta\rfloor\bigr)
      \right\}
  \bigr \rvert
\\
\leq
  r(N\alpha/2^\lambda+2).
\end{multline*}
%}}} carry propagation
\end{lemma}
%%}}} lem_carry (carry propagation lemma)
Let $\mathcal F_n$ the set of rational numbers $p/q$ such that $1\leq q\leq n$,
the \emph{Farey series of order} $n$.
Each $a\in \mathcal F_n$ has two {\em neighbours} $a_L,a_R\in \mathcal F_n$, satisfying $a_L<a<a_R$ and $(a_L,a)\cap \mathcal F_n=(a,a_R)\cap\mathcal F_n=\emptyset$.
We have the following elementary lemma concerning this set
(see~\cite[chapter~3]{HW1954}).
%{{{ lem_farey
\begin{lemma}\label{lem_farey}
Assume that $a/b$, $c/d$ are reduced fractions such that $b,d>0$ and $a/b<c/d$.
Then $a/b<(a+c)/(b+d)<c/d$.
If $a/b$ and $c/d$ are neighbours in the Farey series $\mathcal F_n$, then
$bc-ad=1$ and $b+d>n$,
moreover
\[ (a+c)/(b+d)-a/b < \frac 1{bn}\quad\text{and}\quad c/d-(a+c)/(b+d) < \frac 1{dn}.\]
\end{lemma}
%}}} lem_farey
Let $\alpha\in\mathbb R$ and $Q$ a positive integer.
We assign a fraction $p_Q(\alpha)/q_Q(\alpha)$ to $\alpha$ according to the Farey dissection of the reals:
consider reduced fractions $a/b<c/d$ that are neighbours in the Farey series $\mathcal F_Q$, such that $a/b\leq \alpha<c/d$.
If $\alpha<(a+c)/(b+d)$, then set $p_Q(\alpha)=a$ and $q_Q(\alpha)=b$, otherwise set $p_Q(\alpha)=c$ and $q_Q(\alpha)=d$.
Lemma~\ref{lem_farey} implies
%{{{ eqn_dirichlet
\begin{equation}\label{eqn_dirichlet}
\bigl \lvert q_Q(\alpha)\alpha-p_Q(\alpha)\bigr \rvert < Q^{-1}.
\end{equation}
%}}} eqn_dirichlet
We will call an interval of the form $\{\alpha\in\mathbb R:p_Q(\alpha)=p, q_Q(\alpha)=q\}$ a \emph{Farey interval} around $p/q$.
%}}} Section Lemmas
%{{{ Section proofs
\section{Proof of Propositions~\ref{prp_1} and~\ref{prp_2}}\label{sec_proofs}
As in the proof of Proposition~2.5 in~\cite{MS2017},
it is sufficient to prove that there exists $\eta>0$ and a constant $C$
such that
\[    \frac{S_0(N,2^\nu,\xi)}{N2^\nu}\leq CN^{-\eta}    \]
for all real numbers $\xi$ and
for all positive integers $N$ and $\nu$
such that there exists a real number $D\geq 1$ satisfying
$N^{\rho_1}\leq D\leq N^{\rho_2}$ and $D<2^\nu\leq 2D$,
where $S_0$ is defined according to~\eqref{eqn_S0_def} and~\eqref{eqn_S0_def_continuous}.

In order to treat the two propositions to some extent in parallel,
we will work with two measures $\boldsymbol\mu$:
for Proposition~\ref{prp_1} we take the measure defined by $\boldsymbol\mu(A)=\left\lvert A\cap \mathbb Z\right\rvert$, while for Proposition \ref{prp_1}, $\boldsymbol\mu$ is the Lebesgue measure.
Moreover, we note that in this proof, implied constants in estimates depend only on $m$.

By Cauchy--Schwarz, followed by van der Corput's inequality~\eqref{eqn_vdc} ($R_0$ will be specified later), we obtain
%{{{ eqn
\begin{multline*}
\bigl \lvert S_0(N,2^\nu,\xi)\bigr \rvert^2
\leq
2^\nu
\frac{N+R_0}{R_0}
\int_{2^\nu}^{2^{\nu+1}}
\sup_{\beta\geq 0}
\sum_{0\leq \lvert r_0\rvert<R_0}
\biggl(1-\frac{\lvert r_0\rvert}{R_0}\biggr)
\e\bigl(r_0\xi\bigr)
\\
\times\sum_{\substack{0\leq n<N\\0\leq n+r_0<N}}
\e\Biggl(\frac 12
s\bigl(\left\lfloor(n+r_0)\alpha+\beta\right\rfloor\bigr)
-\frac 12
s\bigl(\left\lfloor n\alpha+\beta\right\rfloor\bigr)
\Biggr)
\,\mathrm d \boldsymbol\mu(\alpha)
\end{multline*}
%}}} eqn

We apply the ``carry propagation lemma'' (Lemma~\ref{lem_carry}),
treat the summand $r_0=0$ separately,
and omit the condition $0\leq n+r_0<N$.
Moreover, we consider $r_0$ and $-r_0$ synchronously.
In this way we obtain for all $\lambda\geq 0$
%{{{ eqn
\begin{multline*}
\bigl \lvert S_0(N,2^\nu,\xi)\bigr \rvert^2
\ll \bigl(2^\nu N\bigr)^2 E_0
+
\frac{2^\nu N}{R_0}
\sum_{1\leq r_0<R_0}
\\
\times\int_{2^\nu}^{2^{\nu+1}}
\sup_{\beta\geq 0}
\Biggl \lvert
\sum_{0\leq n<N}
\e\Biggl(\frac 12
s_\lambda\bigl(\lfloor(n+r_0)\alpha+\beta\rfloor\bigr)
-\frac 12 s_\lambda\bigl(\lfloor n\alpha+\beta\rfloor\bigr)\Biggr)
\Biggr \rvert
\,\mathrm d\boldsymbol\mu(\alpha)
,
\end{multline*}
%}}}eqn
where
\[ E_0=\frac 1{R_0}+\frac{R_0\, 2^\nu}{2^\lambda}+\frac {R_0}N.  \]

We apply Cauchy--Schwarz on the sum over $r_0$ and the integral over $\alpha$ in order to prepare our expression for another application of van der Corput's inequality.
It follows that
%{{{
\begin{equation*}
\bigl \lvert S_0(N,2^\nu,\xi)\bigr \rvert^4
\ll
\frac{2^{3\nu}N^2}{R_0}
\sum_{1\leq r_0<R_0}
\int_{2^\nu}^{2^{\nu+1}}
\sup_{\beta\geq 0}
\bigl \lvert S_1\bigr \rvert^2
\,\mathrm d\boldsymbol\mu(\alpha)
+\bigl(2^\nu N\bigr)^4 E_0
\end{equation*}
%}}}
where
%{{{ eqn_S1_definition
\begin{equation*}%\label{eqn_S1_definition}
S_1
=
\sum_{0\leq n<N}
\e\Biggl(
  \frac 12
  s_\lambda\bigl(\lfloor (n+r_0)\alpha+\beta\rfloor\bigr)
  -\frac 12
  s_\lambda\bigl(\lfloor n\alpha+\beta\rfloor\bigr)
\Biggr).
\end{equation*}
%}}} eqn_S1_definition
(Note that the error term is also squared, but if it is larger or equal to $1$, the estimate is trivial anyway. We will use this argument again in a moment.)
We apply van der Corput's inequality~\eqref{eqn_vdc} with $R=R_1$ and $K=K_1$ to be chosen later:
%{{{ eqn_S1_summand_estimate
\begin{multline*}%\label{eqn_S1_summand_estimate}
\bigl \lvert S_1\bigr \rvert^2
\leq
\frac{N+K_1(R_1-1)}{R_1}
\sum_{0\leq\lvert r_1\rvert<R_1}
\biggl(1-\frac{\lvert r_1\rvert}{R_1}\biggr)
\\
\times
\sum_{\substack{0\leq n<N\\0\leq n+r_1K_1<N}}
\e\Biggl(
  \frac 12
  \sum_{\varepsilon_0,\varepsilon_1\in\{0,1\}}
  s_{\lambda}\bigl(\lfloor(n+\varepsilon_0r_0+\varepsilon_1r_1K_1)\alpha+\beta\rfloor\bigr)
\Biggr),
\end{multline*}
%}}} eqn_S1_summand_estimate
therefore, taking together the summands for $r_1$ and $-r_1$ and omitting the condition $0\leq n+r_1K_1<N$,
%{{{ eqn_
\begin{equation*}%\label{eqn_}
\bigl \lvert S_0(N,2^\nu,\xi)\bigr \rvert^4
\ll
\frac{2^{3\nu}N^3}{R_0\,R_1}
\sum_{\substack{1\leq r_0<R_0\\0\leq r_1<R_1}}
\int_{2^\nu}^{2^{\nu+1}}
\sup_{\beta\geq 0}
\bigl \lvert S_2\bigr \rvert
\,\mathrm d\boldsymbol\mu(\alpha)
+\bigl(2^\nu N\bigr)^4\bigl(E_0+E_1\bigr),
\end{equation*}
%}}}
where
%{{{ definition of S_2
\[
S_2 =
\sum_{0\leq n<N}
\e\Biggl(
  \frac 12
  \sum_{\varepsilon_0,\varepsilon_1\in\{0,1\}}
  s_{\lambda}\bigl(\lfloor(n+\varepsilon_0r_0+\varepsilon_1r_1K_1)\alpha+\beta\rfloor\bigr)
  \Biggr)
\]
%}}} definition of S_4
and
\[  E_1=\frac{R_1K_1}{N}.  \]
Cauchy--Schwarz over $r_0$, $r_1$ and $\alpha$ yields
%{{{ eqn_
\begin{equation*}\label{eqn_}
\bigl \lvert S_0(N,\nu,\xi)\bigr \rvert^8
\ll
\frac{2^{7\nu}N^6}{R_0R_1}
\sum_{\substack{1\leq r_0<R_0\\0\leq r_1<R_1}}
\int_{2^\nu}^{2^{\nu+1}}
\sup_{\beta\geq 0}
\lvert S_2\rvert^2
\,\mathrm d\boldsymbol\mu(\alpha)
+
\bigl(2^\nu N\bigr)^8\bigl(E_0+E_1\bigr).
\end{equation*}
%}}}
We apply van der Corput's inequality with $R=R_2$ and $K=K_2$ to be chosen later:
%{{{eqn
\[ %\label{eqn_S3_summand_estimate}
\frac{\bigl \lvert S_0(N,2^\nu,\xi)\bigr \rvert^8}{
\bigl(2^\nu N\bigr)^8}
\ll
\bigl(E_0+E_1+E_2\bigr)
+\frac{1}{R_0R_1R_2 2^\nu N}
\sum_{\substack{1\leq r_0<R_0\\0\leq r_1<R_1\\0\leq r_2<R_2}}
\int_{2^\nu}^{2^{\nu+1}}
\sup_{\beta\geq 0}
\bigl \lvert
S_3
\bigr \rvert
\,\mathrm d\boldsymbol\mu(\alpha)
\]
%}}}
where
\[
S_3=
\sum_{0\leq n<N}
\e\Biggl(
  \frac 12
  \sum_{\varepsilon_0,\varepsilon_1,\varepsilon_2\in\{0,1\}}
   s_{\lambda}\bigl(\lfloor n\alpha+\beta+\varepsilon_0r_0\alpha+\varepsilon_1r_1K_1\alpha+\varepsilon_2r_2K_2\alpha\rfloor\bigr)
\Biggr)
\]
and
$E_2=R_2K_2/N.$
Continuing in this manner and replacing the range of integration (we note that we are going to choose $\lambda>\nu$ later), we obtain
%{{{eqn
\begin{multline} \label{eqn_S0_S4}
\left \lvert \frac{S_0(N,2^\nu,\xi)}{2^\nu N}\right\rvert^{2^{m+1}}
\ll
\bigl(E_0+E_1+\cdots+E_m\bigr)
\\+\frac{1}{R_0R_1\cdots R_m 2^\nu N}
\sum_{\substack{1\leq r_0<R_0\\0\leq r_i<R_i, 1\leq i\leq m}}
\int_0^{2^\lambda}
\sup_{\beta\geq 0}
\,\bigl \lvert
S_4
\bigr \rvert
\,\mathrm d\boldsymbol\mu(\alpha),
\end{multline}
%}}}
where
%{{{
\begin{multline*}
S_4
=
\sum_{0\leq n<N}
\e\Biggl(
  \frac 12
  \sum_{\varepsilon_0,\ldots,\varepsilon_m\in\{0,1\}}
   s_{\lambda}\bigl(\lfloor n\alpha+\beta+\varepsilon_0r_0\alpha+\varepsilon_1r_1K_1\alpha+\cdots+\varepsilon_mr_mK_m\alpha\rfloor\bigr)
\Biggr)
\end{multline*}
%}}}
and
\begin{align*}
E_0&=\frac 1{R_0}+\frac{R_0\, 2^\nu}{2^\lambda}+\frac {R_0}N,\\
E_i&=\frac{R_i\,K_i}{N}\quad\textrm{for }1\leq i\leq m.
\end{align*}

%{{{ explanations
Now we choose the multiples $K_1,\ldots,K_m$ in such a way that the number of digits to be taken into account is reduced from $\lambda$ to $\rho\coloneqq\lambda-(m+1)\mu$, where $\mu$ is chosen later.
For this we use Farey series, see~\eqref{eqn_dirichlet}.
%}}} explanations
Let
\begin{align*}
K_1&=q_{2^{2\mu+2\sigma}}\left(\frac{\alpha}{2^{2\mu}}\right) q_{2^\sigma}\biggl(\frac{p_{2^{2\mu+2\sigma}}\left(\alpha/2^{2\mu}\right)}{2^{(m-1)\mu}}\biggr);
\\
K_i&=q_{2^{\mu+2\sigma}}\left(\frac{\alpha}{2^{(i+1)\mu}}\right) q_{2^\sigma}\biggl(\frac{p_{2^{\mu+2\sigma}}\left(\alpha/2^{(i+1)\mu}\right)}{2^{(m-i)\mu}}\biggr)\quad\textrm{for }2\leq i<m;
\\
K_m&=q_{2^{\mu+\sigma}}\left(\frac{\alpha}{2^{(m+1)\mu}}\right),
\end{align*}
where $\sigma$ is chosen later.
Moreover, we set
\begin{align*}
M_1&=p_{2^{2\mu+2\sigma}}\biggl(\frac{\alpha}{2^{2\mu}}\biggr)
q_{2^\sigma}\left(
\frac{p_{2^{2\mu+2\sigma}}\left(\alpha/2^{2\mu}\right)}
{2^{(m-1)\mu}}
\right)
;\\
M_i&=p_{2^{\mu+2\sigma}}\biggl(\frac{\alpha}{2^{(i+1)\mu}}\biggr)
q_{2^\sigma}\left(
\frac{p_{2^{\mu+2\sigma}}\left(\alpha/2^{(i+1)\mu}\right)}
{2^{(m-i)\mu}}
\right)
\quad\textrm{for }2\leq i<m;
\\
M_m&=p_{2^{\mu+\sigma}}\biggl(\frac{\alpha}{2^{(m+1)\mu}}\biggr)
.
\end{align*}

By Lemma~\ref{lem_farey}, estimating the second factor in the definition of $K_i$ and $M_i$ by $2^\sigma$, we have

\begin{align}
\bigl \lvert K_1\alpha-2^{2\mu}M_1\bigr \rvert&<2^{-\sigma};\nonumber\\
\biggl \lvert \frac{K_i\alpha}{2^{i\mu}}-2^\mu M_i\biggr \rvert&<2^{-\sigma}\quad\textrm{for }2\leq i<m;\label{eqn_approx}
\\
\biggl \lvert
\frac{K_m\alpha}{2^{m\mu}}-2^\mu M_m\biggr \rvert&<2^{-\sigma}.\nonumber
\end{align}

We are going to use these inequalities in order to replace $r_iK_i\alpha$ in the sum $S_4$,
starting with $r_1K_1\alpha$.
We treat the case that $\alpha$ is an integer first:
in this case, $K_1\alpha=2^{2\mu}M_1$,
and by the fact that the arguments of $s_\lambda$ corresponding to $\varepsilon_1=0,1$ differ by a multiple of $2^{2\mu}$ we may shift the argument by $2\mu$ digits and thus reduce the number of digits to be taken into account from $\lambda$ to $\lambda-2\mu$.
%{{{
\begin{multline*}
S_4
=
\sum_{0\leq n<N}
\e\Biggl(
  \frac 12
  \sum_{\varepsilon_0,\ldots,\varepsilon_m\in\{0,1\}}
   s_{2\mu,\lambda}\left(\left\lfloor n\alpha+\beta
   \right.\right.
   \\
  \left.\left.
  +\varepsilon_0r_0\alpha
   +\varepsilon_1r_1M_12^{2\mu}
  +\varepsilon_2r_2K_2\alpha
   +\cdots+\varepsilon_mr_mK_m\alpha\right\rfloor\right)
\Biggr)
\\=
\sum_{0\leq n<N}
\e\Biggl(
  \frac 12
  \sum_{\varepsilon_0,\ldots,\varepsilon_m\in\{0,1\}}
   s_{\lambda-2\mu}\biggl(\biggl\lfloor \frac{n\alpha+\beta}{2^{2\mu}}
  +\frac{\varepsilon_0r_0\alpha}{2^{2\mu}}
+\varepsilon_1r_1M_1
  +\frac{\varepsilon_2r_2K_2\alpha}{2^{2\mu}}
   +\cdots+\frac{\varepsilon_mr_mK_m\alpha}{2^{2\mu}}\biggr\rfloor\biggr)
\Biggr).
\end{multline*}
%}}}
In the case $\alpha\not\in \mathbb Z$,
we use the inequalities~\eqref{eqn_approx} and the argument that $n\alpha$-sequences are usually not close to an integer.
This can be made precise as follows.
Assume that
%{{{ eqn_spaced
\begin{equation}\label{eqn_spaced}
\lVert n\alpha+\beta'\rVert\geq R_1/2^\sigma,
\end{equation}
%}}} eqn_spaced

where $\beta'=\beta+\varepsilon_0r_0\alpha+\varepsilon_2r_2K_2\alpha+\cdots+\varepsilon_mr_mK_m\alpha$,
and that $2R_1<2^\sigma$.
Using the inequality~\eqref{item_fpf_3} in Lemma \ref{lem_fractional_part_facts} with $\varepsilon=1/2^\sigma$, where $\sigma\geq 1$ is chosen later,
and~\eqref{eqn_dirichlet}, we obtain
\[    \bigl \langle r_1K_1\alpha\bigr \rangle=r_1\bigl\langle K_1\alpha\bigr\rangle=r_12^{2\mu} M_1.    \]
Applying~\eqref{item_fpf_1}, setting $\varepsilon=R_1/2^\sigma$, we see that~\eqref{eqn_spaced} together with~\eqref{eqn_approx} implies
%{{{ eqn_floor_split
\begin{equation*}%\label{eqn_floor_split}
\lfloor n\alpha+r_1K_1\alpha+\beta'\rfloor
=
\lfloor n\alpha+r_12^{2\mu}M_1+\beta'\rfloor.
\end{equation*}
%}}} eqn_floor_split
The number of $n$ where hypothesis~\eqref{eqn_spaced} fails for some $\varepsilon_0,\varepsilon_2,\ldots,\varepsilon_m$ can be estimated by discrepancy estimates for $\{n\alpha\}$-sequences:
for all positive integers $N$ and $2R_1<2^\sigma$ we have
%{{{ eqn_count_close_integers (application to counting the "bad" integers)
\begin{equation*}%\label{eqn_count_close_integers}
\begin{aligned}
\hspace{4em}&\hspace{-4em}\bigl \lvert\bigl\{n\in [0,N-1]:\lVert n\alpha+\beta'\rVert\leq R_1/2^\sigma\bigr\}\bigr \rvert\\
&=\bigl \lvert\bigl\{n\in [0,N-1]:n\alpha+\beta'\in \left[-R_1/2^\sigma,R_1/2^\sigma\right]+\mathbb Z\bigr\}\bigr \rvert\\
&=\bigl \lvert\bigl\{n\in [0,N-1]:n\alpha\in \left[0,2R_1/2^\sigma\right]-\beta'-R_1/2^\sigma+\mathbb Z\bigr\}\bigr \rvert\\
&\leq ND_N(\alpha)+2R_1N/2^\sigma.
\end{aligned}
\end{equation*}
%}}} eqn_count_close_integers (application to counting the "bad" integers)
Therefore, the number of $n\in[0,N-1]$ such that
$\lVert n\alpha+\beta'\rVert\leq R_1/2^\sigma$ for some
$\varepsilon_0,\varepsilon_2,\ldots,\varepsilon_m\in\{0,1\}$
is bounded by
$2^mN\bigl(D_N(\alpha)+2R_1/2^\sigma\bigr)$, which is
$\ll N\bigl(D_N(\alpha)+2R_1/2^\sigma\bigr)$ by our convention that implied constants may depend on $m$.

We replace $K_1\alpha$ by $2^{2\mu}M_1$ and subsequently shift the digits by $2\mu$ and obtain
%{{{
\begin{multline*}
S_4
=
\mathcal O\bigl(N D_N(\alpha)+NR_1/2^\sigma\bigr)
+
\sum_{0\leq n<N}
\e\Biggl(
  \frac 12
  \sum_{\varepsilon_0,\ldots,\varepsilon_m\in\{0,1\}}
   s_{\lambda-2\mu}\biggl(\biggl\lfloor \frac{n\alpha+\beta}{2^{2\mu}}
   \\
  +\frac{\varepsilon_0r_0\alpha}{2^{2\mu}}
   +\varepsilon_1r_1M_1
  +\frac{\varepsilon_2r_2K_2\alpha}{2^{2\mu}}
   +\cdots+\frac{\varepsilon_mr_mK_m\alpha}{2^{2\mu}}\biggr\rfloor\biggr)
\Biggr).
\end{multline*}

Repeating this argument for all $i\in\{2,\ldots,m\}$, we obtain
\begin{multline*}
S_4=N\mathcal O\left(\widetilde D_N(\alpha)+D_N\left(\frac{\alpha}{2^{2\mu}}\right)+\cdots+D_N\left(\frac{\alpha}{2^{m\mu}}\right)+\frac{R_1+\cdots+R_m}{2^\sigma}\right)
\\+
\sum_{0\leq n<N}
\e\Biggl(
  \frac 12
  \sum_{\varepsilon_1,\ldots,\varepsilon_m\in\{0,1\}}
   s_{\lambda-(m+1)\mu}\biggl(\biggl\lfloor \frac{n\alpha+\beta}{2^{(m+1)\mu}}
+   \frac{\varepsilon_0r_0\alpha}{2^{(m+1)\mu}}
  +\sum_{1\leq i\leq m}
   \frac{\varepsilon_ir_iM_i}{2^{(m-i)\mu}}
\biggr\rfloor\biggr)
\Biggr),
\end{multline*}
%}}}
where $\widetilde D_N(\alpha)=D_N(\alpha)$ if $\alpha\not\in\mathbb Z$ and $\widetilde D_N(\alpha)=0$ otherwise.

Now the second factor in the definition of $K_i$ comes into play.
We use the definition of $M_i$ together with the approximation property~\eqref{eqn_dirichlet}, and apply the discrepancy estimate for $n\alpha$-sequences again to obtain
\begin{equation}\label{eqn_S4_S5}
S_4=N\mathcal O\left(\widetilde D_N(\alpha)+D_N\left(\frac{\alpha}{2^{2\mu}}\right)+\cdots+D_N\left(\frac{\alpha}{2^{(m+1)\mu}}\right)+\frac{R_1+\cdots+R_m}{2^\sigma}\right)+S_5,
\end{equation}
where
%{{{
\[S_5=
\sum_{0\leq n<N}
\e\Biggl(
  \frac 12
  \sum_{\varepsilon_0,\ldots,\varepsilon_m\in\{0,1\}}
   s_{\lambda-(m+1)\mu}\biggl(\biggl\lfloor \frac{n\alpha+\beta}{2^{(m+1)\mu}}+\frac{\varepsilon_0r_0\alpha}{2^{(m+1)\mu}}\biggr\rfloor
+\sum_{1\leq i\leq m}
\varepsilon_ir_i\mathfrak p_i
\biggr)\Biggr),
\]
%}}}
and
\begin{align}
\mathfrak p_1&=
p_{2^\sigma}\left(
\frac{p_{2^{2\mu+2\sigma}}\left(\alpha/2^{2\mu}\right)}
{2^{(m-1)\mu}}
\right);\nonumber
\\
\mathfrak p_i&=
   p_{2^\sigma}\left(
\frac{p_{2^{\mu+2\sigma}}\left(\alpha/2^{(i+1)\mu}\right)}
{2^{(m-i)\mu}}
\right)
\quad\textrm{for }2\leq i<m;\label{eqn_pi_definition}
\\
\mathfrak p_m&=
p_{2^{\mu+\sigma}}\left(\frac{\alpha}{2^{(m+1)\mu}}\right).\nonumber
\end{align}

Our next goal is to remove the Beatty sequence occurring in $S_5$, and also to remove the integers $\mathfrak p_i$.
The resulting expression can be handled by the Gowers norm estimate given in Proposition~\ref{prp_gowers}, which will finish the proof.

We start by splitting the Beatty sequence into two summands.
Let $t,T$ be integers such that $0\leq t<T$ and define
%{{{ definition of tilde S_5
\begin{multline*}
S_6 %(a,d,r,s,t,N,R,T,\mu,\lambda)
=
\sum_{\substack{0\leq n<N\\\frac tT\leq\left\{\frac{n\alpha+\beta}{2^{(m+1)\mu}}\right\}< \frac{t+1}T}}
\e\Biggl(
  \frac 12
  \sum_{\varepsilon_0,\ldots,\varepsilon_m\in\{0,1\}}
   s_{\lambda-(m+1)\mu}\biggl(\biggl\lfloor \frac{n\alpha+\beta+\varepsilon_0r_0\alpha}{2^{(m+1)\mu}}\biggr\rfloor
+\sum_{1\leq i\leq m}
\varepsilon_ir_i\mathfrak p_i
\biggr)\Biggr).
\end{multline*}
%}}} definition of tilde S_5
We define
%{{{ definition of G
\[
%(
G
=
\biggl\{1\leq t<T:
  \left[
    \frac tT+\frac {\varepsilon_0r_0\alpha}{2^{(m+1)\mu}},
    \frac{t+1}T+\frac{\varepsilon_0r_0\alpha}{2^{(m+1)\mu}}
  \right)
\cap \mathbb Z= \emptyset
\biggr\}.
%]
\]
%}}} definition of G
Clearly we have $\lvert G\rvert\geq T-2$, since we have to exclude at most one $t$.
For $t\in\{0,\ldots,T-1\}\setminus G$ we estimate $S_6$ trivially, using Lemma~\ref{lem_f_discrepancy}: we obtain
%{{{ eqn_S6_trivial_estimate
\begin{equation}\label{eqn_S6_trivial_estimate}
S_6\ll \frac NT
+
ND_N\left(\frac {\alpha}{2^{(m+1)\mu} }\right)
.
\end{equation}
%}}} eqn_S6_trivial_estimate
Assume that $t\in G$ and that $t/T\leq \{(n\alpha+\beta)/2^{(m+1)\mu}\}<(t+1)/T$.
Then \[\left\lfloor \frac{n\alpha+\beta}{2^{(m+1)\mu}}\right\rfloor+\frac tT+\frac{\varepsilon_0r_0\alpha}{2^{(m+1)\mu}}\leq \frac{n\alpha+\beta+\varepsilon_0r_0\alpha}{2^{(m+1)\mu}}<
\left\lfloor \frac{n\alpha+\beta}{2^{(m+1)\mu}}\right\rfloor+\frac {t+1}T+\frac{\varepsilon_0r_0\alpha}{2^{(m+1)\mu}}\]
and the assumption $t\in G$ gives
\[
\left\lfloor\frac{n\alpha+\beta+\varepsilon_0r_0\alpha}{2^{(m+1)\mu} }\right\rfloor
=
\left\lfloor\frac{n\alpha+\beta}{2^{(m+1)\mu} }\right\rfloor+\left\lfloor\frac tT+\frac{\varepsilon_0r_0\alpha}{2^{(m+1)\mu} }\right\rfloor
\]
for $\varepsilon_0\in\{0,1\}$.
From these observations we obtain for $t\in G$:
%{{{ eqn_S6_nontrivial_estimate
\begin{multline*}
S_6
=
\sum_{0\leq k<2^\rho}
\sum_{\substack{0\leq n<N\\
\frac tT\leq\left\{ \frac{n\alpha+\beta}{2^{(m+1)\mu}}\right\} <\frac{t+1}T\\
\left\lfloor\frac{n\alpha+\beta}{2^{(m+1)\mu} }\right\rfloor\equiv k\bmod 2^\rho
}}
\e\Biggl(
  \frac 12
 \sum_{\varepsilon_0,\ldots,\varepsilon_m\in\{0,1\}}
   s_\rho\biggl(
   k+\left\lfloor \frac tT+\frac{\varepsilon_0r_0\alpha}{2^{(m+1)\mu}}\right\rfloor
+\sum_{1\leq i\leq m}
\varepsilon_ir_i\mathfrak p_i
\biggr)\Biggr)
.
\end{multline*}
%}}} eqn_tilde_S5_nontrivial_estimate
Note that the Beatty sequence $\lfloor (n\alpha+\beta)/2^{(m+1)\mu}\rfloor$ does not occur in the summand any more.
We may therefore remove the second summation by estimating the number of times the three conditions under the summation sign are satisfied.
At this point we want to stress the fact that $N$ is going to be significantly larger than $2^\rho=2^{\lambda-(m+1)\mu}$.
Using Lemma~\ref{lem_f_discrepancy} and the usually very small discrepancy of $n\alpha$-sequences, this fact will enable us to remove the summation over $n$, while introducing only a negligible error term for most $\alpha$.
This is the point in the proof where the successive ``cutting away'' of binary digits with the help of Farey series pays off.

By Lemma~\ref{lem_f_discrepancy}, applied with $K=2^\rho$,
and noting that $\lambda=(m+1)\mu+\rho$, we obtain for $t\in G$
%{{{ eqn_S6_relevant_estimate
\begin{equation}\label{eqn_S6_relevant_estimate}
S_6
=
\frac{N}{2^\rho T}
\,S_7
+O\left(2^\rho ND_N\left(\frac \alpha{2^\lambda} \right)\right),
\end{equation}
%}}} eqn_S6_relevant_estimate
where
\[S_7=
\sum_{0\leq k<2^\rho}
\e\Biggl(
  \frac 12
  \sum_{\varepsilon_0,\ldots,\varepsilon_m\in\{0,1\}}
   s_\rho\biggl(
   k+\left\lfloor\frac tT+\frac{\varepsilon_0r_0\alpha}{2^{(m+1)\mu}}\right\rfloor
+\sum_{1\leq i\leq m}
\varepsilon_ir_i\mathfrak p_i
\biggr)\Biggr).
\]
We note the important fact that this expression is independent of $\beta$.
This will allow us to remove the maximum over $\beta$ inside the integral over $\alpha$,
and thus prove the strong statement on the level of distribution.

We wish to simplify this expression in such a way that Proposition~\ref{prp_gowers} is applicable.
To this end, we use the summation over $r_i$ and the integral over $\alpha$.
We define
\[
S_8=
\int_0^{2^\lambda}
\sum_{\substack{0\leq r_1,\ldots,r_m<2^\rho}}
\bigl \lvert S_7\bigr \rvert
\,\mathrm d\boldsymbol\mu(\alpha),
\]
which is an expression that will appear when we expand the original sum $S_0$.

We are going to apply the argument that for most $\alpha<2^\lambda$ (with respect to $\boldsymbol \mu$) the $2$-valuation of $\mathfrak p_1,\ldots,\mathfrak p_m$ is small.
For these $\alpha$, the term $r_i\mathfrak p_i\bmod 2^\rho$ attains each $k\in\{0,\ldots,2^\rho-1\}$ not too often, as $r_i$ runs. We may therefore replace $r_i\mathfrak p_1$ by $r_i$ and thus obtain full sums over $r_i$ (we note that we will set $R_i=2^\rho$ for $1\leq i\leq m$).
In order to make this argument work, we are going to utilize the following technical result, the proof of which we give in section~\ref{sec_proofofdivisibility}.
%{{{ lem_exceptions
\begin{lemma}\label{lem_exceptions}
Let $\mu,\lambda,\sigma,\gamma,m$ be nonnegative integers such that $m\geq 2$ and

\begin{equation}\label{eqn_gamma_conditions}
\begin{aligned}
\lambda\geq (m+1)\mu,&\quad\gamma\leq \lambda-(m+1)\mu,\\
\mu\geq 4\sigma,&\quad\sigma\geq \gamma\geq 1.
\end{aligned}
\end{equation}
Let $\mathfrak p_1,\ldots,\mathfrak p_m$ be defined by~\eqref{eqn_pi_definition} and set
\[A=\{\alpha\in\{0,\ldots,2^\lambda-1\}:2^{3\gamma}\mid \mathfrak p_i\textrm{ for some }i=1,\ldots,m\}.\]
Then \[\lvert A\rvert=\mathcal O\bigl(2^{\lambda-\gamma}\bigr).\]
Analogously, if
\[A=\{\alpha\in[0,2^\lambda]:2^{3\gamma}\mid \mathfrak p_i\textrm{ for some }i=1,\ldots,m\}.\]
Then \[\boldsymbol\lambda(A)=\mathcal O\bigl(2^{\lambda-\gamma}\bigr),\]
where $\boldsymbol \lambda$ is the Lebesgue measure.
The implied constants are independent of $\mu,\lambda,\sigma$, and $\gamma$.
\end{lemma}
%}}}

Let $A$ be defined as in this lemma.
We choose $R_i=2^\rho$ for $1\leq i\leq m$.

Assume that $\alpha\not\in A$.
Then by an elementary argument, $r_i\mathfrak p_i\bmod 2^\rho$ attains each value not more than $2^{3\gamma}$ times, as $r_i$ runs through $\{0,\ldots,2^\rho-1\}$.
The contribution for $\alpha\in A$ will be estimated trivially by the lemma.
We obtain
\begin{equation*}%\label{eqn_sumS8S9}
S_8
\leq
\mathcal O\left(2^{\lambda+(m+1)\rho-\gamma}\right)
+
2^{3\gamma m}
\int_0^{2^\lambda}
\sum_{0\leq r_1,\ldots,r_m<2^\rho}
\left\lvert
S_9
\right\rvert
\,\mathrm d\boldsymbol\mu(\alpha),
\end{equation*}

where
\[
S_9=
\sum_{0\leq n<2^\rho}
\e\Biggl(
  \frac 12
  \sum_{\varepsilon_0,\ldots,\varepsilon_m\in\{0,1\}}
   s_\rho\biggl(
   n+\left\lfloor\frac tT+\frac{\varepsilon_0r_0\alpha}{2^{(m+1)\mu}}\right\rfloor
+\sum_{1\leq i\leq k} \varepsilon_ir_i
\biggr) \Biggr).
\]

The next step is removing the remaining floor function, using the integral over $\alpha$.
In the continuous case, the expression $\left\lfloor t/T+r_0K_0\alpha/2^{(m+1)\mu}\right\rfloor\bmod 2^\rho$  runs through $\{0,\ldots,2^\rho-1\}$ in a completely uniform manner.
That is, for $r_0\neq 0$ we have
\[\boldsymbol\lambda\left(\left\{\alpha\in[0,2^\lambda]:
\left\lfloor t/T+r_0\alpha/2^{(m+1)\mu}\right\rfloor\equiv k\bmod 2^\rho\right\}\right)=2^{\lambda-\rho},\]
where $\boldsymbol\lambda$ is the Lebesgue measure.
We consider the discrete case.
Assume that $r_0\leq 2^{(m+1)\mu}$ (we will choose $R_0$ very small at the end of the proof, so that this will be satisfied).
Then the set of $\alpha\in\{0,\ldots,2^\lambda-1\}$
such that $\lfloor t/T+r_0\alpha/2^{(m+1)\mu}\rfloor\equiv k\bmod 2^\rho$
decomposes into at most $r_0+1$ many intervals (note that $\lambda=(m+1)\mu+\rho$), each having $\leq 2^{(m+1)\mu}/r_0+1$ elements.
In total we have $\ll 2^{\lambda-\rho}$ elements, where the implied constant is absolute.
It follows that
\begin{align*}
S_8\ll
2^{\lambda+(m+1)\rho-\gamma}
+
2^{\lambda-\rho+3\gamma m}
\sum_{0\leq r_0,\ldots,r_m<2^\rho}
\lvert S_{10}(r_0,\ldots,r_m)\rvert,
\end{align*}
where
\[S_{10}(r_0,\ldots,r_m)
=
\sum_{0\leq n<2^\rho}
\e\left(\frac 12\sum_{\varepsilon_0,\ldots,\varepsilon_m\in\{0,1\}}
s_\rho\left(n+\sum_{0\leq i\leq m}\varepsilon_i r_i\right)\right).
\]

As a final step in the procedure of reducing the main theorems to Proposition~\ref{prp_gowers}, we are going to to remove the absolute value around $S_{10}$.
For brevity, we set
\[g(n)=
\sum_{\varepsilon_0,\ldots,\varepsilon_m\in\{0,1\}}
s_\rho\left(n+\sum_{0\leq i\leq m}\varepsilon_i r_i\right)\]
By the $2^\rho$-periodicity of $g$ we have
\begin{align*}
\hspace{3em}&\hspace{-3em}
\sum_{0\leq r_0,\ldots,r_m<2^\rho}\lvert S_{10}(r_0,\ldots,r_m)\rvert^2
=
\sum_{0\leq r_0,\ldots,r_m<2^\rho}
\sum_{0\leq n_1,n_2<2^\rho}
\e\left(\frac 12
g(n_1)+\frac 12g(n_2)\right)
\\&=
\sum_{0\leq r_0,\ldots,r_m<2^\rho}
\sum_{0\leq n_1<2^\rho}
\sum_{0\leq r_{m+1}<2^\rho}
\e\left(\frac 12
g(n_1)+\frac 12g(n_1+r_{m+1})
\right)
\\&=
\sum_{0\leq r_0,\ldots,r_{m+1}<2^\rho}
\sum_{0\leq n_1<2^\rho}
\e\left(\frac 12
g(n_1)+\frac 12g(n_1+r_{m+1})
\right)
\\&=
\sum_{0\leq r_0,\ldots,r_{m+1}<2^\rho}
\sum_{0\leq n_1<2^\rho}
\e\left(\frac 12
\sum_{\varepsilon_0,\ldots,\varepsilon_m\in\{0,1\}}
\sum_{\varepsilon_{m+1}\in\{0,1\}}
s_\rho(n_1+\varepsilon\cdot r+\varepsilon_{m+1}r_{m+1})
\right)
\\&=
\sum_{0\leq r_0,\ldots,r_{m+1}<2^\rho}
S_{10}(r_0,\ldots,r_{m+1}).
\end{align*}
We have therefore removed the absolute value around $S_{10}$ for the price an additional variable $r_{m+1}$.
This means that we have reduced our main theorems to Proposition~\ref{prp_gowers}.

By this Proposition and Cauchy-Schwarz we obtain
\begin{equation}\label{eqn_S8_estimate}
S_8\ll 2^{\lambda+(m+1)\rho}\left(2^{-\gamma}+2^{3\gamma m-\eta\rho}\right)
\end{equation}
for some $\eta>0$.

It remains to collect the error terms and to choose values for the free variables.
Using \eqref{eqn_S6_relevant_estimate} and~\eqref{eqn_S6_trivial_estimate},
we obtain
\begin{multline*}
S_5\ll
\sum_{t\not\in G}\mathcal O\left(\frac NT+ND_N\left(\frac{\alpha}{2^{(m+1)\mu}}\right)\right)
+\sum_{t\in G}\left(\frac N{2^\rho T}S_7+\mathcal O\left(2^\rho N D_N\left(\frac{\alpha}{2^\lambda}\right)\right)\right)
\\=
\frac N{2^\rho T}\sum_{t\in G}S_7
+\mathcal O\left(\frac NT+ND_N\left(\frac{\alpha}{2^{(m+1)\mu}}\right)+2^\rho N TD_N\left(\frac{\alpha}{2^\lambda}\right)\right)
\end{multline*}
and by~\eqref{eqn_S4_S5} and~\eqref{eqn_S0_S4} we obtain

\begin{multline}\label{eqn_S0_errorterms}
\left\lvert\frac{S_0(N,\nu,\xi)}{2^\nu N}\right\rvert^{2^{m+1} }
\ll
\mathcal O\left(\frac 1{R_0}+\frac{R_0\,2^\nu}{2^\lambda}+\frac{R_0}{N}+\frac{R_1K_1}{N}\cdots+\frac{R_mK_m}{N}\right)
\\+
\frac 1{2^\nu N}
\int_0^{2^\lambda}
N\mathcal O\left(\widetilde D_N(\alpha)+D_N\left(\frac{\alpha}{2^{2\mu}}\right)+\cdots+D_N\left(\frac{\alpha}{2^{(m+1)\mu}}\right)+\frac{R_1+\cdots+R_m}{2^\sigma}\right)
\,\mathrm d\boldsymbol\mu(\alpha),
\\+
\frac 1{2^\nu N}
\int_0^{2^\lambda}
\mathcal O\left(\frac NT+ND_N\left(\frac{\alpha}{2^{(m+1)\mu}}\right)+2^\rho NT D_N\left(\frac{\alpha}{2^\lambda}\right)\right)
\,\mathrm d\boldsymbol\mu(\alpha),
\\+\frac 1{R_0\cdots R_m2^\nu N}
\frac N{2^\rho T}
\sum_{t\in G}
\sum_{1\leq r_0<R_0}
\int_0^{2^\lambda}
\sum_{\substack{0\leq r_1,\ldots,r_m<2^\rho}}
\bigl \lvert S_7\bigr \rvert
\,\mathrm d\boldsymbol\mu(\alpha).
%\\=F_1+F_2+F_3+S_8.
\end{multline}

We employ the mean discrepancy estimates from Lemma~\ref{lem_mean_discrepancy}.
Assume that $\delta\leq \lambda$.
In the continuous case we have
\[
\frac 1{2^\nu}\int_0^{2^\lambda}
D_N\left(\frac \alpha{2^\delta}\right)
\,\mathrm d\alpha
\ll
2^{\lambda-\nu-\delta}
\int_0^{2^\delta}
D_N\left(\frac \alpha{2^\delta}\right)
\,\mathrm d\alpha
\ll
2^{\lambda-\nu}\frac{(\logp N)^2}N,
\]
while the discrete case gives
\[
\frac 1{2^\nu}\sum_{0\leq d<2^\lambda}
D_N\left(\frac d{2^\delta}\right)
\ll
2^{\lambda-\delta-\nu}
\frac{N+2^\delta}N(\logp N)^2
=2^{\lambda-\nu}(\logp N)^2
\left(\frac 1N+\frac 1{2^\delta}\right)
\]
In total, noting that $\lambda\geq (m+1)\mu$,
the discrepancy terms can be estimated by
\[\ll
2^{\lambda-\nu}(\logp N)^2 2^\rho T
\left(\frac 1N+\frac 1{2^{2\mu}}\right).
\]
By~\eqref{eqn_S8_estimate}, the last summand in~\eqref{eqn_S0_errorterms} can be estimated by
\[
\ll
2^{\lambda-\nu}\left(2^{-\gamma}+2^{3\gamma m-\eta\rho}\right).
\]
Moreover, using the facts $R_1=\cdots=R_m=2^\rho$
and $K_i\leq 2^{2\mu+3\sigma}$ for $1\leq i\leq m$,
we obtain
\begin{multline}\label{eqn_S0_errorterms2}
\left\lvert\frac{S_0(N,\nu,\xi)}{2^\nu N}\right\rvert^{2^{m+1} }
\ll
\frac 1{R_0}+\frac{R_0\,2^\nu}{2^\lambda}+\frac{R_0}{N}+\frac{2^{\rho+2\mu+3\sigma}}{N}+
\\
2^{\lambda-\nu}(\logp N)^2 2^\rho
T\left(\frac 1N+\frac 1{2^{2\mu}}\right)
+2^{\rho-\sigma+\lambda-\nu}
+\frac 1T
+2^{\lambda-\nu}\left(2^{-\gamma}+2^{3\gamma m-\eta\rho}\right)
\end{multline}
with some implied constant only depending on $m$.
Collecting also the requirements on the variables we assumed in the course of our calculation, we see that this estimate is valid as long as
\begin{equation}\label{eqn_requirements}
\begin{aligned}
&R_0,T\geq 1, m\geq 2, \gamma, \nu,\lambda,\rho,\mu\geq 0,&&
R_1=\cdots=R_m=2^\rho,\\
&\lambda>\nu,&&
\rho=\lambda-(m+1)\mu,\\
&\gamma\leq \rho<\sigma-1,
&&\mu\geq 4\sigma,\\
&R_0\leq 2^{(m+1)\mu}.
\end{aligned}
\end{equation}

It remains to choose the variables within these constraints.
Choose the integer $j\geq 1$ in such a way that $N^{j-1}\leq 2^\nu<N^j$ and set
$m=3j-1$. Clearly, $m\geq 2$.
We define
\[\mu=\left\lfloor\frac{\nu}{m+1+1/8}\right\rfloor,\quad
\sigma=\lfloor \mu/4\rfloor,\quad \widetilde\rho=\nu-(m+1)\mu.\]
We obtain the inequalities
$N\geq 2^{3\mu}$, $\mu\geq 4\sigma$, $\widetilde\rho\geq 0$.
Moreover, for large $\nu$ we obtain $\widetilde\rho\sim \mu/8$.

Choose
$\gamma=\lfloor \widetilde\rho\eta/(6m)\rfloor$ and $R_0=\lfloor 2^{\gamma/4}\rfloor$.
Then the last summand in~\eqref{eqn_S0_errorterms2} is $\ll 2^{\lambda-\nu}\bigl(2^{-\gamma}+2^{-\widetilde\rho\eta/2}\bigl)\ll 2^{\lambda-\nu-\gamma}$.
Finally, set $\lambda=\nu+\lfloor\gamma/2\rfloor$, $T=2^\gamma$ and $\rho=\lambda-(m+1)\mu$.
It follows that $\rho=\widetilde \rho+\lfloor \gamma/2\rfloor \sim \frac \mu 8(1+\eta/(12m))\leq \mu/8+\mu/192$.
Using these definitions, it is not hard to see that, for large $N$ and $\nu$, the requirements~\eqref{eqn_requirements} are met.

Moreover, using the statements $N^{\rho_1}\leq D\leq N^{\rho_2}$ and $D<2^\nu\leq 2D$
we can easily estimate~\eqref{eqn_S0_errorterms2} term by term and conclude that $S_0(N,\nu,\xi)/(2^\nu N)\leq C N^{-\eta'}$ for some $\eta'>0$ and some constant $C$. This finishes the proof of Propositions~\ref{prp_1} and~\ref{prp_2} and therefore of our main theorems.
%{{{
It remains to prove our auxiliary results.
\subsection{Proof of Proposition~\ref{prp_gowers}}
We utilize ideas from the paper~\cite{K2017} by Konieczny.
Set
\[A_\rho(\mathbf a)=
\frac 1{2^{(m+1)\rho}}
\sum_{\substack{0\leq n<2^\rho\\0\leq r_1,\ldots,r_m<2^\rho}}
\e\left(\frac 12\sum_{\varepsilon\in\{0,1\}^m}s_\rho(n+\varepsilon\cdot r+\mathbf a_\varepsilon)\right).
\]
Then in analogy to equation~(16) of~\cite{K2017}, we get after a similar calculation (using $m\geq 2$)
\begin{equation}\label{eqn_konieczny_rec}
A_{\rho+1}(\mathbf a)=\frac{(-1)^{\lvert \mathbf a\rvert}}{2^{m+1}}
\sum_{e_0,\ldots,e_m\in\{0,1\}} A_\rho(\delta(\mathbf a,e)),
\end{equation}
where $\lvert \mathbf a\rvert = \sum_{\varepsilon\in\{0,1\}^m} \mathbf a_\varepsilon$
and
\[\delta(\mathbf a,e)_\varepsilon=
\left\lfloor \frac{\mathbf a_\varepsilon+e_0+\sum_{1\leq i\leq m}\varepsilon_ie_i}2\right\rfloor.\]

We define a directed graph with weighted edges according to~\eqref{eqn_konieczny_rec}.
The set of vertices is given by the set of families $\mathbf a\in\mathbb Z^{\{0,1\}^m}$.
There is an edge from $\mathbf a$ to $\mathbf b$ if and only if there is an $e=(e_0,\ldots,e_m)\in\{0,1\}^{m+1}$ such that $\delta(\mathbf a,e)=\mathbf b$ and this edge has the weight
\[
w(\mathbf a,\mathbf b)
=\frac{(-1)^{\lvert \mathbf a\rvert}}{2^{m+1}}
\left\lvert\left\{
e\in\{0,1\}^{m+1}:\delta(\mathbf a,e)=\mathbf b
\right\}\right\rvert.
\]
Note that
\begin{equation}\label{eqn_total_weight}
\sum_{\mathbf b\in\mathbb Z^{\{0,1\}^m}}\lvert w(\mathbf a,\mathbf b)\rvert=1,
\end{equation}
which we will need later.
We are interested in the subgraph $(V,E,w)$ induced by the set of vertices reachable from $\mathbf 0$.
This graph is finite: we have
\[\max_{\varepsilon\in\{0,1\}^m}\lvert \delta(\mathbf a,e)_\varepsilon\rvert
\leq
\frac 12\left(\max_{\varepsilon\in\{0,1\}^m}\lvert \mathbf a_\varepsilon\rvert
+m+1\right)\]
and by induction, it follows that
$\max_{\varepsilon\in\{0,1\}^m}\lvert \mathbf a_\varepsilon\rvert<m+1$ for all $\mathbf a\in V$, which implies the finiteness of $V$.

Moreover, this subgraph is strongly connected. We prove this by showing that $\mathbf 0$ is reachable from each $\mathbf a\in V$.
This follows immediately by considering the path $(\mathbf a=\mathbf a^{(0)},\mathbf a^{(1)}),\ldots,(\mathbf a^{(k)},\mathbf a^{(k+1)})$ defined by $\mathbf a^{(j+1)}=\delta(\mathbf a^{(k)},(0,\ldots,0))$. It is clear from the definition of $\delta$ that such a path reaches $\mathbf 0$ if $k$ is large enough.

We wish to apply~\eqref{eqn_konieczny_rec} recursively.
We therefore define, for two vertices $\mathbf a, \mathbf b\in V$ and a positive integer $k$, the weight $w_k(\mathbf a,\mathbf b)$ as the sum of all weights of paths of length $k$ from $\mathbf a$ to $\mathbf b$.
(Here the weight of a path is the product of the weights of the edges.)

In order to prove Proposition~\ref{prp_gowers},
it is sufficient to prove that there is a $k$ such that
\[\sum_{\mathbf b\in V}\lvert w_k(\mathbf a,\mathbf b)\rvert <1\]
for all $\mathbf a\in V$.
In order to prove this, it is sufficient, by the strong connectedness of the graph and~\eqref{eqn_total_weight}, to prove that there are two paths of the same length from $\mathbf 0$ to $\mathbf 0$ such that their respective weights have different sign.
One of this paths is the trivial one, choosing $e_0=\cdots=e_m=0$ in each step. This path has positive weight.

For the second path, we follow Konieczny~\cite[proof of Proposition 2.3]{K2017}.
As in that paper, we define $\mathbf a^{(0)}=\mathbf a^{(m+1)}=\mathbf 0$ and for $1\leq j\leq m$,
\[\mathbf a^{(j)}_\varepsilon=\begin{cases}1,&\textrm{if }\varepsilon_1=\cdots=\varepsilon_j=1;\\0,&\textrm{otherwise.}\end{cases}\]
Assuming for a moment that there is an edge from $\mathbf a^{(j)}$ to $\mathbf a^{(j+1)}$ for all $j\in\{0,\ldots,m\}$,
it is easy to see that each edge $(\mathbf a^{(j)},\mathbf a^{(j+1)})$ has positive weight for $0\leq j<m$, while $(\mathbf a^{(m)},\mathbf a^{(m+1)})$ has negative weight. Proving that these vertices indeed define a path is contained completely in the argument given by Konieczny.
This finishes the proof of Lemma~\ref{prp_gowers}.
%}}}

\subsection{Proof of Lemma~\ref{lem_exceptions}}
\label{sec_proofofdivisibility}
We choose an integer $\gamma>0$ and bound the size of the set of $\alpha<2^\lambda$ such that $2^{3\gamma}\mid \mathfrak p_i$ for some $i\in\{1,\ldots,m\}$.
We will need the following two lemmas.
%{{{lem_farey_divisibility1
\begin{lemma}\label{lem_farey_divisibility1}
%First case: continuous $\alpha$.
Let $\boldsymbol\lambda$ be the Lebesgue measure.
Assume that $K\geq 1$ and $\gamma\geq 0$ are integers.
Then
\[\boldsymbol\lambda
\bigl(\{x\in[0,1]:2^\gamma\mid q_K(x)\}\bigr)
\ll \frac 1{2^\gamma}+\frac 1K.\]
The constant in this estimate is absolute.
\end{lemma}
\begin{proof}
We have to sum up the lengths of the Farey intervals around $p/q$ such that $2^\gamma\mid q$.
By Lemma~\ref{lem_farey}, each such fraction contributes at most $2/(Kq)$.
By summing over $p\in\{1,\ldots,q\}$, this gives a contribution $2/K$ for each multiple $q$ of $2^\gamma$,
and we obtain a total contribution
\[
\ll \sum_{\substack{1\leq q\leq K\\2^\gamma\mid q}}
\frac 1K
\leq \frac 1{2^\gamma}+\frac 1K.
\]
\end{proof}
%}}}
%{{{lem_farey_divisibility2
\begin{lemma}\label{lem_farey_divisibility2}
%Second case: discrete $\alpha$.
  Let $x_0,\ldots,x_{M-1}\in[0,1]$ and $\delta>0$.
Assume that $\lVert x_i-x_j\rVert\geq \delta$ for $i\neq j$.
Then
\[
\bigl \lvert \{n\in \{0,\ldots,M-1\}:2^\gamma\mid q_K(x_i)\} \bigr \rvert
\ll
\frac{K^2}{2^\gamma}+\frac 1\delta\left(\frac 1{2^\gamma}+\frac 1K\right).
\]
The implied constant is absolute.
\end{lemma}
\begin{proof}
In each Farey interval around $p/q$ such that $q$ is divisible by $2^\gamma$ there are at most $2/(Kq\delta)+1$ many points $x_i$.
By summing over $p$ and $q$, we can bound the number of points in such intervals by
\begin{multline*}
\ll \sum_{\substack{1\leq q\leq K\\2^\gamma\mid q}}
\sum_{1\leq p\leq q}
\left(\frac 1{qK\delta}+1\right)
=
\sum_{\substack{1\leq q\leq K\\2^\gamma\mid q}}
\left(\frac 1{K\delta}+q\right)
=
\left(K2^{-\gamma}+1\right)\frac 1{K\delta}+
\sum_{\substack{1\leq q\leq K\\2^\gamma\mid q}}q
\\\leq
\frac 1{2^\gamma \delta}+\frac 1{K\delta}
+
2^\gamma \sum_{1\leq q'\leq\lfloor K2^{-\gamma} \rfloor }q'
\ll
\frac{K^2}{2^\gamma}+\frac 1{2^\gamma\delta}+\frac 1{K\delta}.
\qedhere
\end{multline*}
\end{proof}
%}}}

We proceed to the proof of Lemma~\ref{lem_exceptions}.
Consider $\mathfrak p_1$ and the case ``$\alpha$ discrete''.
In this case, we have $p_{2^{2\mu+2\sigma}}(\alpha/2^{2\mu})=\alpha$.
Assume therefore that $\alpha=\alpha_0+2^{(m-1)\mu}\alpha_1$, where $\alpha_0\in\{0,\ldots,2^{(m-1)\mu}-1\}$ and $\alpha_1\in \{0,\ldots,2^{\lambda-(m-1)\mu}-1\}$.

Then
\[\mathfrak p_1=p_{2^\sigma}\bigl(\alpha/2^{(m-1)\mu}\bigr)
=p_{2^\sigma}\bigl(\alpha_0/2^{(m-1)\mu}\bigr)+
q_{2^\sigma}\bigl(\alpha_0/2^{(m-1)\mu}\bigr)\alpha_1.
\]
By Lemma~\ref{lem_farey_divisibility2}, using also~\eqref{eqn_gamma_conditions}, it follows that the number of $\alpha_0\in\{0,\ldots,2^{(m-1)\mu}-1\}$ such that $2^\gamma\nmid q_{2^\sigma}\bigl(\alpha_0/2^{(m-1)\mu}\bigr)$ is
$2^{(m-1)\mu}\left(1-\mathcal O(2^{-\gamma})\right)$.
For each such $\alpha_0$, we let $\alpha_1$ run through $\{0,\ldots,2^{\lambda-(m-1)\mu}-1\}$.
Then two occurrences $\alpha_1$, $\alpha_1'$ such that $2^{2\gamma}\mid\mathfrak p_1$ are separated by at least $2^\gamma$ steps;
it follows that the number of such $\alpha_1$ is bounded by $2^{\lambda-(m-1)\mu-\gamma}$.
Putting these errors together, we see that the number of $\alpha\in\{0,\ldots,2^\lambda-1\}$ such that $2^{2\gamma}\nmid \mathfrak p_1$ is
given by $2^{(m-1)\mu}\left(1-\mathcal O(2^{-\gamma})\right)2^{\lambda-(m-1)\mu}\left(1-\mathcal O(2^{-\gamma})\right)= 2^\lambda\left(1-\mathcal O(2^{-\gamma})\right)$.

Next, we consider the continuous case.
We write $\alpha=\alpha_0+2^{2\mu}\alpha_1+2^{(m+1)\mu}\alpha_2$ ,
where                      %(
$\alpha_0\in [0,2^{2\mu})$ %]
is real and $\alpha_1<2^{(m-1)\mu}$ and $\alpha_2<2^{\lambda-(m+1)\mu}$ are nonnegative integers.
Set $p=p_{2^{2\mu+2\sigma}}\bigl(\alpha_0/2^{2\mu}\bigr)$ and
$q=q_{2^{2\mu+2\sigma}}\bigl(\alpha_0/2^{2\mu}\bigr)$.
Then
\begin{align*}
\frac{p_{2^{2\mu+2\sigma}}\bigl(\alpha/2^{2\mu}\bigr)}{2^{(m-1)\mu}}
&=
\frac{p+\bigl(\alpha_1+2^{(m-1)\mu}\alpha_2\bigr)q}{2^{(m-1)\mu}}
=
\frac{p+\alpha_1q}{2^{(m-1)\mu}}+\alpha_2q.
\end{align*}

By the approximation property~\eqref{eqn_dirichlet}
(note that $\sigma\geq 1$) we have
\begin{align*}
\mathfrak p_1&=
\left\langle
\left(\frac{p+\alpha_1q}{2^{(m-1)\mu}}+\alpha_2q\right)
q_{2^\sigma}
\left(\frac{p+\alpha_1q}{2^{(m-1)\mu}}\right)
\right\rangle
\\&=
\left\langle
\frac{p+\alpha_1q}{2^{(m-1)\mu}}
q_{2^\sigma}
\left(\frac{p+\alpha_1q}{2^{(m-1)\mu}}\right)
\right\rangle
+\alpha_2q\,
q_{2^\sigma}\left(\frac{p+\alpha_1q}{2^{(m-1)\mu}}\right)
\end{align*}
and we note that the first summand does not depend on $\alpha_2$.

As $\alpha_0$ runs through $[0,2^{2\mu}]$, we have by Lemma~\ref{lem_farey_divisibility1} $2^\gamma\nmid q$ in a set of measure $2^{2\mu}(1-\mathcal O(2^{-\gamma}+2^{-2\mu-2\sigma}))$.
By~\eqref{eqn_gamma_conditions}, this is
$2^{2\mu}\bigl(1-\mathcal O\bigl(2^{-\gamma}\bigr)\bigr)$.
Assume that $\alpha_0$ is such that $2^\gamma\nmid q$
and set $\gamma'=\nu_2(q)<\gamma$.
Next, we let $\alpha_1$ run.
We choose $x_j=\bigl\{(p+jq)/2^{(m-1)\mu}\bigr\}$
for $0\leq j<2^{(m-1)\mu-\gamma'}$ and we note that these points satisfy
$\lVert x_i-x_j\rVert\geq 1/2^{(m-1)\mu-\gamma'}$ for $i\neq j$.
By Lemma~\ref{lem_farey_divisibility2} it follows that
\[
\left\{
\alpha_1\in\{0,\ldots,2^{(m-1)\mu-\gamma'}-1\}:
2^\gamma\mid
q_{2^\sigma}\left(\frac{p+\alpha_1 q}{2^{(m-1)\mu}}\right)
\right\}
\ll
\frac{2^{2\sigma}}{2^{\gamma}}+2^{(m-1)\mu-\gamma'}\left(\frac 1{2^\gamma}+\frac 1{2^{\sigma}}\right).
\]
By~\eqref{eqn_gamma_conditions}, this is $\ll 2^{(m-1)\mu-\gamma'-\gamma}$.
Performing this also for the other intervals of length $2^{(m-1)\mu-\gamma'}$, we obtain
\[
\left\{
\alpha_1\in\{0,\ldots,2^{(m-1)\mu}-1\}:
2^\gamma\mid
q_{2^\sigma}\left(\frac{p+\alpha_1 q}{2^{(m-1)\mu}}\right)
\right\}
\ll
2^{(m-1)\mu-\gamma}.
\]

Finally, $\alpha_2$ runs through $\{0,\ldots,2^{\lambda-(m+1)\mu}-1\}$ and we consider $\mathfrak p_1$.
For given good $\alpha_1$ and $\alpha_0$ (such that $2^\gamma\nmid q$ and $2^\gamma\nmid q_{2^\sigma}((p+\alpha_1 q)/2^{(m-1)\mu})$),
$\mathfrak p_1$ is an arithmetic progression in $\alpha_2$ whose common difference is not divisible by $2^{2\gamma}$.
Similarly to the discrete case, it follows that $\mathfrak p_1$ is divisible by $2^{3\gamma}$ for at most
$2^{\lambda-(m+1)\mu-\gamma}$ many $\alpha_2$.
It follows that there is a set of measure
\[
2^{2\mu}\bigl(1-\mathcal O(2^{-\gamma})\bigr)
2^{(m-1)\mu}\bigl(1-\mathcal O(2^{-\gamma})\bigr)
2^{\lambda-(m+1)\mu}\bigl(1-\mathcal O(2^{-\gamma})\bigr)
=2^\lambda\bigl(1-\mathcal O(2^{-\gamma})\bigr)
\]
of $\alpha<2^\lambda$ such that $2^{3\gamma}\nmid \mathfrak p_1$.

The cases $2\leq i\leq m$ do not require any new ideas;
we only give a sketch of a proof.
Let $2\leq i<m$. We treat the discrete and continuous cases in parallel.
We write $\alpha=\alpha_0+2^{(i+1)\mu}\alpha_1+2^{(m+1)\mu}\alpha_2$,
where
$\alpha_0<2^{(i+1)\mu}$,
and
$\alpha_1<2^{(m-i)\mu}$ and $\alpha_2<2^{\lambda-(m+1)\mu}$ are nonnegative integers.
Set $p=p_{2^{\mu+2\sigma}}\bigl(\alpha_0/2^{(i+1)\mu}\bigr)$ and
$q=q_{2^{\mu+2\sigma}}\bigl(\alpha_0/2^{(i+1)\mu}\bigr)$.
Then
%\begin{align*}
%\frac{p_{2^{\mu+2\sigma}}\bigl(\alpha/2^{(i+1)\mu}\bigr)}{2^{(m-i)\mu}}
%&=
%\frac{p+\bigl(\alpha_1+2^{(m-i)\mu}\alpha_2\bigr)q}{2^{(m-i)\mu}}
%=
%\frac{p+\alpha_1q}{2^{(m-1)\mu}}+\alpha_2q
%\end{align*}
%and
\begin{align*}
\mathfrak p_i&=
\left\langle
\frac{p+\alpha_1q}{2^{(m-i)\mu}}
q_{2^\sigma}
\left(\frac{p+\alpha_1q}{2^{(m-i)\mu}}\right)
\right\rangle
+\alpha_2q\,
q_{2^\sigma}\left(\frac{p+\alpha_1q}{2^{(m-i)\mu}}\right),
\end{align*}
as before.
By Lemmas~\ref{lem_farey_divisibility1} and~\ref{lem_farey_divisibility2} we have $2^\gamma\nmid q$
for $\alpha_0$ in a set of measure $2^{(i+1)\mu}(1-\mathcal O(2^{-\gamma}))$,
where we used $2\mu+4\sigma\leq (i+1)\mu$ in the discrete case.
(We note that this last inequality is the reason for defining $\mathfrak p_1$ separately, using $2^{2\mu}$ instead of $2^\mu$.)
The remaining steps are as before, and this case is finished.

Finally, in the case $i=m$ we write
$\alpha=\alpha_0+2^{(m+1)\mu}\alpha_1$, where
$\alpha_0<(m+1)\mu$ and
$\alpha_1\in\{0,\ldots,2^{\lambda-(m+1)\mu}-1\}$.
Then
\[\mathfrak p_m=p_{2^{\mu+\sigma}}\bigl(\alpha_0/2^{(m+1)\mu}\bigr)+q_{2^{\mu+\sigma}}\bigl(\alpha_0/2^{(m+1)\mu}\bigr)\alpha_1.\]

By Lemmas~\ref{lem_farey_divisibility1} and~\ref{lem_farey_divisibility2} and~\eqref{eqn_gamma_conditions} we have $2^\gamma\mid q_{2^{\mu+\sigma}}\bigl(\alpha_0/2^{(m+1)\mu}\bigr)$ for $\alpha_0$ in a set of measure $\mathcal O(2^{(m+1)\mu-\gamma})$ and the statement follows as before.

In total, we have a set of measure $2^\lambda\bigl(1-\mathcal O(2^{-\gamma})\bigr)$ of $\alpha<2^\lambda$ such that $2^{3\gamma}\nmid \mathfrak p_i$ for all $i$.
%}}} Section proofs
%{{{ Acknowledgements
\section*{Acknowledgements}
The author wishes to thank Thomas Stoll for helpful discussions during his stay in Nancy, where the work on this project began.
Moreover, the author wishes to thank Michael Drmota and Clemens M\"ullner for several fruitful discussions on the topic.
Finally, the author is indebted to Etienne Fouvry for valuable advice.
%}}} Acknowledgements
%{{{ Bibliography
\bibliographystyle{siam}
\bibliography{lod}
%}}} Bibliography
\end{document}